\documentclass[a4paper,final]{amsart}
\usepackage[margin=1.3in]{geometry}

\usepackage{amsmath}%
\usepackage{amsfonts}%
\usepackage{amssymb}%
\usepackage{amsthm}
\usepackage{graphicx}
\usepackage{dsfont}
\usepackage{pdfcomment}
\usepackage{tikz}
\usepackage{tikz-3dplot}
\usepackage{subfigure}
\usepackage{soul}
\usepackage{enumerate}
\usepackage{cancel}
\usepackage{stmaryrd}

\usepackage{hyperref}

\usetikzlibrary{decorations.markings}
\newtheorem{theorem}{Theorem}[section]

\newtheorem{conjecture}{Conjecture}
\newtheorem{corollary}[theorem]{Corollary}

\newtheorem{definition}[theorem]{Definition}

\allowdisplaybreaks

\newtheorem{lemma}[theorem]{Lemma}

\newtheorem{proposition}[theorem]{Proposition}
\newtheorem{remark}[theorem]{Remark}

\DeclareMathOperator{\Arg}{Arg}

\DeclareMathOperator{\cratio}{cr}

\DeclareMathOperator{\Ker}{Ker}

\DeclareMathOperator{\Real}{Re}
\DeclareMathOperator{\Imaginary}{Im}
\DeclareMathOperator{\grad}{grad}

\DeclareMathOperator{\hol}{hol}
\DeclareMathOperator{\Hom}{Hom}
\renewcommand{\Re}{\Real}
\renewcommand{\Im}{\Imaginary}
\def\C{\mathbb{C}}

\title{Quadratic differentials and circle patterns on complex projective tori}
\author{Wai Yeung Lam}
\thanks{This work was partially supported by the ANR/FNR project SoS, INTER/ANR/16/11554412/SoS,
	ANR-17-CE40-0033.}
\address{Mathematics Research Unit, Université du Luxembourg, L-4364 Esch-sur-Alzette} \email{wai.lam@uni.lu}
\begin{document}
		
\begin{abstract}
Given a triangulation of a closed surface, we consider a cross ratio system that assigns a complex number to every edge satisfying certain polynomial equations per vertex. Every cross ratio system induces a complex projective structure together with a circle pattern. In particular, there is an associated conformal structure. We show that for any triangulated torus, the projection from the space of cross ratio systems with prescribed Delaunay angles to the Teichm\"{u}ller space of the closed torus is a covering map with at most one branch point. Our approach is based on a notion of discrete holomorphic quadratic differentials.
\end{abstract}

\maketitle

\section{Introduction} \label{sec:introduction}

Discrete differential geometry concerns structure-preserving discretizations in differential geometry \cite{Bobenko2008}. Its goal is to establish a discrete theory with rich mathematical structures such that the smooth theory arises in the limit of refinement. It has stimulated applications in computational architecture and computer graphics. To obtain a structured discrete theory, a main challenge is to decide on the right properties to be preserved under discretization. 

A prominent example in discrete conformal geometry is Thurston's circle packing \cite{Stephenson2005}. In the classical theory, holomorphic functions are conformal, mapping infinitesimal circles to themselves. Instead of infinitesimal size, a circle packing is a configuration of finite-size circles where certain pairs are mutually tangent. Thurston proposed regarding the map induced from two circle packings with the same tangency pattern as a discrete holomorphic function. Using machinery from hyperbolic 3-manifolds, a discrete analogue of the Riemann mapping is deduced from Koebe-Andreev-Thurston theorem. Rodin and Sullivan \cite{Rodin1987} showed that it converges to the classical Riemann mapping in the case of hexagonal circle packings as the mesh size tends to zero. By considering general combinatorics instead of restricting to the hexagonal mesh, circle packings can approximate quasi-conformal maps proving the measurable Riemann mapping theorem \cite{He1990,Williams2019}. A natural question is how this theory can be extended to Riemann surfaces and how the convergence depends on the combinatorics.

On the other hand, computer scientists have been using circle packings to approximate conformal maps between surfaces for years \cite{Boris2006}. Among many numerical schemes, circle packings have the advantage of its discrete nature ready for numerical computation and its effective visualization of conformal stretching. It is believed that in the limit of a suitable refinement, the map induced from circle packings would converge to a classical conformal map. It motivates a systematic study of the interplay between circle packings, combinatorics and conformal structures on surfaces.

From the viewpoint of discrete differential geometry, it is advantageous to develop a structured discrete theory that relates successful examples like circle packings. In the previous works \cite{Lam2016, Lam2017, Lam2015a}, we developed a notion of discrete holomorphic quadratic differentials from circle packings and connected it to several discrete theories: discrete harmonic functions, discrete integrable systems and discrete minimal surfaces in space. It is intriguing if this notion could be related to the classical Teichm\"{u}ller theory as like as its smooth counterpart, e.g. parameterizing the space of complex projective structures on a Riemann surface via the Schwarzian derivative.

This article investigates circle patterns on surfaces with complex projective structures, where circle patterns with prescribed intersection angles play a role of a fixed discrete conformal structure. A circle pattern in the plane is simply a realization of a planar graph such that each face has a circumcircle passing through the vertices. Any two circles from adjacent faces thus intersect at a certain angle. For circle patterns on surfaces, they can be formulated in terms of an algebraic system as follows: We denote $M=(V, E, F)$ a triangulation of a closed oriented surface where $V$, $E$ and $F$ are the sets of vertices, edges and faces respectively. Vertices are denoted by $i,j,k$. An unoriented edge is denoted by $\{ij\}=\{ji\}$ indicating its end points are vertices $i$ and $ j$, where $i=j$ is allowed and in that case the edge has to form a non-contractable loop on the surface. Given a realization $z:V \to \mathbb{C}\cup\{\infty\}$ on the Riemann sphere, we associate a complex cross ratio to every common edge $\{ij\}$ shared by triangles $\{ijk\}$ and $\{jil\}$:
\[
cr_{ij} :=  -\frac{(z_k - z_i)(z_l -z_j)}{(z_i - z_l)(z_j - z_k)}
\]
which encodes how the circumdisk of triangle $z_iz_jz_k$ is glued to that of $z_jz_iz_l$. It defines a function  $\cratio: E  \to \mathbb{C}$ satisfying certain polynomial equations:
\begin{definition}\label{def:crsys}
	Suppose $M=(V,E,F)$ is a triangulation of a closed oriented surface. 
	A cross ratio system on $M$ is an assignment $\cratio:E \to \mathbb{C}$ such that for every vertex $i$ with adjacent vertices numbered as $1$, $2$, ..., $n$ in the clockwise order counted from the link of $i$,
	\begin{gather}
	 \Pi_{j=1}^n \cratio_{ij} =1  \label{eq:crproduct}\\
	 \cratio_{i1} + \cratio_{i1} \cratio_{i2} + \cratio_{i1}\cratio_{i2}\cratio_{i3} + \dots +  \cratio_{i1}\cratio_{i2}\dots\cratio_{in} =0 \label{eq:crsum}
	\end{gather}
	where $\cratio_{ij} = \cratio_{ji}$. Equivalently,
	\[
	\Pi_{j=1}^n \left( \begin{array}{cc}
	\cratio_{ij} & 1 \\
	0 & 1
	\end{array} \right) = \left( \begin{array}{cc}
	1 & 0 \\
	0 & 1
	\end{array} \right).
	\]
The cross ratio system is called \underline{Delaunay} if

\noindent
(a) the argument of the cross ratios $\Arg \cratio$ takes value in $[0, \pi )$,

\noindent
(b) the graph  $(V, E_+)$ is the 1-skeleton of a CW decomposition of $M$ where  \[E_{+}=\{ ij  \in E|  \Arg \cratio_{ij} \neq 0 \}.\]

\noindent
Furthermore  the ramification index $s:V\to \mathbb{N}$ is defined by
\[
2\pi s_i =  \sum_j \Arg \cratio_{ij}.
\] 
Without further notice, we focus on cross ratio systems without branch points (i.e. $s\equiv 1$).
\end{definition}

A cross ratio system provides a recipe to glue neighboring circumdisks, which resembles to Thurston's equations for gluing ideal hyperbolic polyhedra \cite{Thurston1982}. Equations \eqref{eq:crproduct} and \eqref{eq:crsum} ensures that the holonomy around each vertex under the gluing construction is trivial and have been considered by Fock and Gonchrov \cite{Fock1993,Fock2006}. Notice that cross ratios are invariant under complex projective transformations (i.e. M\"{o}bius transformations). One can deduce that every Delaunay cross ratio system defines a complex projective structure on the closed surface $M$ (see Section \ref{sec:cp1}). It yields a forgetful map $f: \mathcal{D} \to P(M)$ from the space of all Delaunay cross ratio systems $\mathcal{D}$ to the space of all marked complex projective structures $P(M)$.

In this article, we are interested in a subset of cross ratio systems that have prescribed arguments. These correspond to circle patterns where neighbouring circles intersect at prescribed angles, which is a natural generalization of circle packings \cite{Lam2017}.

\begin{definition} Suppose $M$ is a closed triangulated surface. A Delaunay angle structure is an assignment of angles $\Theta:E \to [0,\pi)$ with $\Theta_{ij}=\Theta_{ji}$ satisfying the following:   
	\begin{enumerate}[(i)]
		\item For every vertex $i$,  \[ \sum_j \Theta_{ij} = 2\pi\] where the sum is taken over the neighboring vertices of $i$ on the universal cover.
		\item For any collection of edges $(e_0,e_1,e_2,\dots,e_n=e_0)$ whose dual edges form a simple closed contractable path on the surface, then
		\[
		\sum_{i=1}^{n} \Theta_{e_i} > 2\pi
		\]
		unless the path encloses exactly one primal vertex.
	\end{enumerate}
  We write $P(\Theta)\subset \mathcal{D}$ the space of all Delaunay cross ratio systems $\cratio$ with $\Arg \cratio \equiv \Theta$.
\end{definition}    

 The space $P(\Theta)$ consists of all circle patterns on complex projective surfaces with intersection angles $\Theta$ and is known to be nonempty. By the discrete uniformization theorem  \cite{Bobenko2004, Rivin}, $P(\Theta)$ contains exactly one cross ratio system whose underlying complex projective structure can be reduced to an Euclidean structure if $g=1$ and a hyperbolic structure if $g>1$. 

It is known that every complex projective structure induces a developing map of the universal cover to the Riemann sphere. Its holonomy yields a representation of the fundamental group in $PSL(2,\mathbb{C})$ up to conjugation, whose quotient \[
\mathcal{X}(M):= \Hom( \pi_1(M), PSL(2,\mathbb{C}))\sslash PSL(2,\mathbb{C}) 
\]
in the sense of geometric invariant theory is called a character variety. We have a holonomy map
\[
 \hol: P(M) \to \mathcal{X}(M).
\]
 In the case of complex projective tori, $\hol(P(M))$ in $\mathcal{X}(M)$ is smooth everywhere except at the image of Euclidean structures. 

\begin{theorem}\label{thm:holo} For any Delaunay angle structure $\Theta$ on a torus,  $P(\Theta)$ is a real analytic surface homeomorphic to $\mathbb{R}^2$. Furthermore,  the holonomy map
	\[ \hol\circ f: P(\Theta) \to \mathcal{X}(M)\]
	is an embedding which passes through the image of Euclidean tori exactly once. Particularly the forgetful map $f|_{P(\Theta)}$ is an embedding into the space of marked complex projective structures $P(M)$.
\end{theorem}

We denote $\mathcal{T}(M)$ the Teichm\"{u}ller space, i.e. the space of marked conformal structures. If $M$ is a torus, $\mathcal{T}(M)$ is a manifold homeomorphic to $\mathbb{R}^2$. We further write $\pi: P(M) \to \mathcal{T}(M)$ from the space of complex projective structures to the Teichm\"{u}ller space.

\begin{theorem} \label{thm:covering} For a Delaunay angle structure $\Theta$ on a torus, the projection \[ \pi\circ f:P(\Theta)\backslash \{\cratio_0\} \to  \mathcal{T}(M)\backslash \{\tau_0\}\] is a finite-sheet covering map where $\cratio_0$ is the unique cross ratio system in $P(\Theta)$ that induces an Euclidean torus and $\tau_0$ is the associated conformal structure. In particular, the projection \[ \pi\circ f:P(\Theta)\to  \mathcal{T}(M)\] is a covering map with at most one branch point. 
\end{theorem}

Both the proofs of Theorems \ref{thm:holo} and \ref{thm:covering} rely on Rivin's results in \cite{Rivin} and our notion of discrete holomorphic quadratic differentials interpreted as tangent vectors of $P(\Theta)$ (See Section \ref{sec:hqd}).

It remains a conjecture whether $\pi\circ f|_{P(\Theta)}$ is a diffeomorphism for the torus and hence $f(P(\Theta))$ is a section of the fiber bundle $\pi: P(M) \to \mathcal{T}(M)$. The analogues of Theorems \ref{thm:holo} and \ref{thm:covering}  for surfaces with genus $g>1$ are still open. In particular it is interesting to investigate whether $f(P(\Theta))$ becomes tangential to the fiber of $\pi$ under a triangulation refinement, which would imply the convergence of discrete holomorphic quadratic differentials to their smooth counterpart.

Circle patterns on tori have appeared in various forms \cite{Sass2012}. Doyle's spiral circle packings \cite{Beardon1994} as a discrete analogue of exponential functions can be seen as a developing map of a circle packing on an affine torus, which is further extended to discretize Painlevé equations from integrable systems \cite{Agafonov2000}. On the other hand, recently we established a correspondence between circle patterns on tori and the dimer models from statistical mechanics \cite{fKen2018}. It is intriguing how they would interact with the underlying conformal structures. See Figure \ref{fig:circlepatterns} for examples of circle patterns on the torus.

Our result also indicates that circle patterns on surfaces play a role of discrete quasi-conformal map. Fixing two conformal structures on the torus and Delaunay angles, Theorem \ref{thm:covering} indicates that there always exists circle patterns on complex projective tori with the prescribed conformal structures. Under a suitable refinement of the triangulation, the map induced by such circle patterns could converge to a quasi-conformal map. It would be interesting to explore how the triangulations affect the Beltrami coefficient, which measures the conformal distortion of the resulting quasi-conformal map.

The organization of the paper is as follows: In Section \ref{sec:background}, we explain the derivation of cross ratio systems and the Delaunay condition. We further introduce discrete holomorphic quadratic differentials as tangent vectors of $P(\Theta)$ and recall their connections to discrete harmonic functions from previous results. In Section \ref{sec:cp1tori}, we focus on the torus. Using complex affine developing maps, we simplify the equations for cross ratio systems. We then apply discrete harmonic functions and Rivin's results to deduce Theorems \ref{thm:holo} and \ref{thm:covering}. In Section \ref{sec:examples}, we explore a couple of examples and show how the above theorems fail for non-Delaunay cross ratio systems. In Section \ref{sec:discussion}, we discuss some open problems.
\subsection{Related work}

Our work is closely related to Kojima, Mizushima and Tan \cite{Kojima2003} who proposed to consider circle packings on surfaces with complex projective structures. They conjectured that this configuration space is a manifold whose projection to the Teichm\"{u}ller space is a diffeomorphism. Here we outline the progress so far:

\begin{enumerate}[(i)]
	\item Torus ($g=1$): Mizushima \cite{Mizushima2000} proved that for the one-vertex triangulation, the projection of the space of circles packings on complex projective tori to the Teichm\"{u}ller space is a diffeomorphism by explicit computation. Kojima, Mizushima and Tan \cite{Kojima2003} showed that in the space of circle packings on any given triangulation, the circle packing on an Euclidean tori is contained in a neighborhood homeomorphic to $\mathbb{R}^2$. However, little is known about the global structure of the configuration space and its projection to the Teichm\"{u}ller space.
		\item Surfaces with genus $g>1$: Kojima, Mizushima and Tan \cite{Kojima2003} similarly showed that in the space of circle packings on any given triangulation, the circle packing on a hyperbolic surface is contained in a neighborhood homeomorphic to $\mathbb{R}^{6g-6}$. Recently, Schlenker and Yarmola \cite{Schlenker2018} showed that in the setting of circle patterns, the projection to the Teichm\"{u}ller space is a proper map. However it remains unknown whether the space of circle patterns with fixed intersection angles is a manifold and whether its projection to the Teich\"{u}ller space is an immersion.
\end{enumerate}

Our contribution is to study the configuration space of Delaunay circle patterns $P(\Theta)$ on tori with arbitrary triangulations and their projection to the Teichm\"{u}ller space. We show that the configuration space is a manifold and its projection is proper everywhere as well as an immersion almost everywhere. Our use of discrete holomorphic quadratic differentials as tangent vectors of $P(\Theta)$ is novel, which could be applicable to surfaces with genus $g>1$. 

\begin{figure}[h!] \centering
	\begin{minipage}{0.49\textwidth}
		\includegraphics[width=1\textwidth]{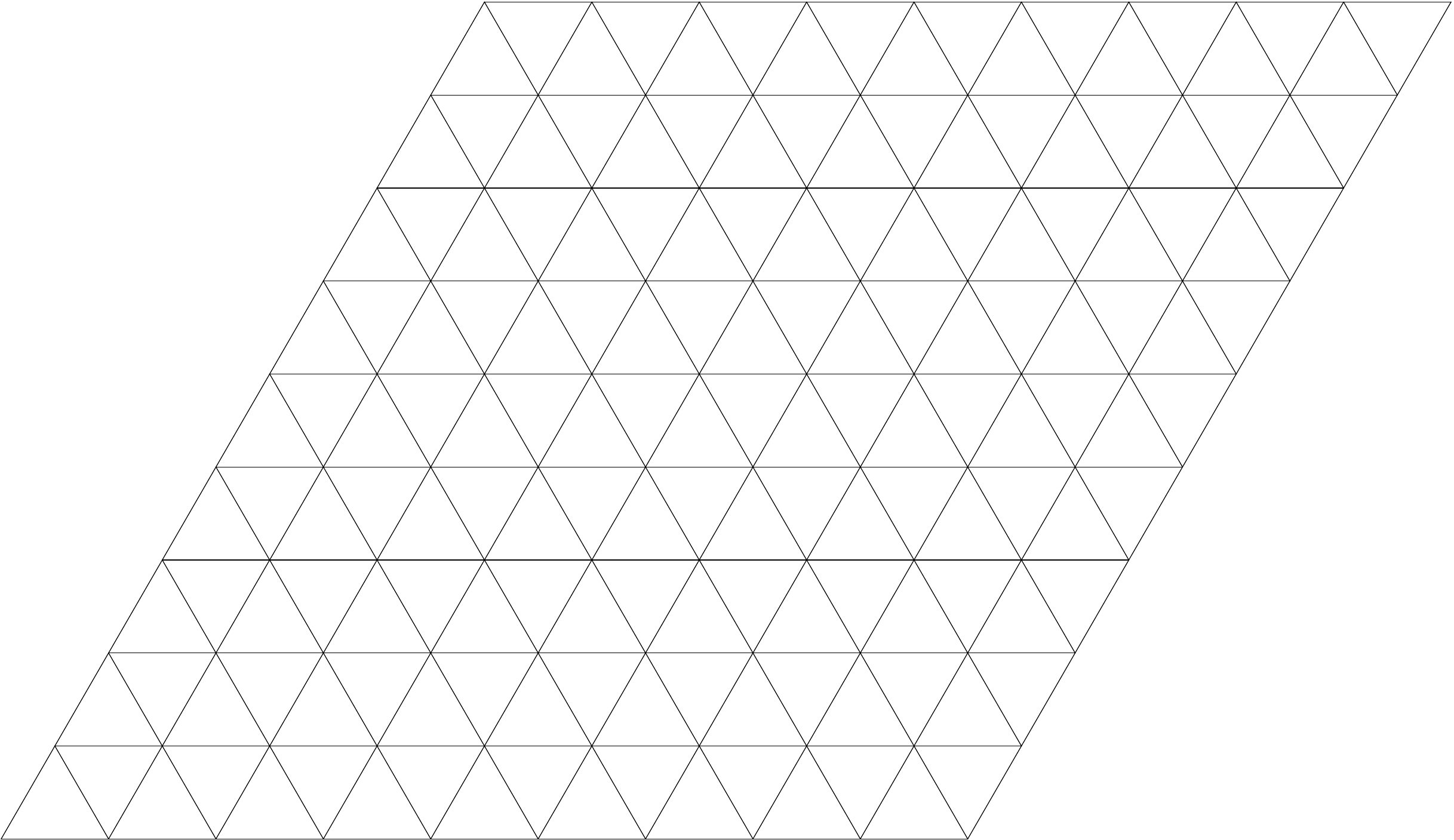}
	\end{minipage}
	\begin{minipage}{0.49\textwidth}
		\includegraphics[width=1\textwidth]{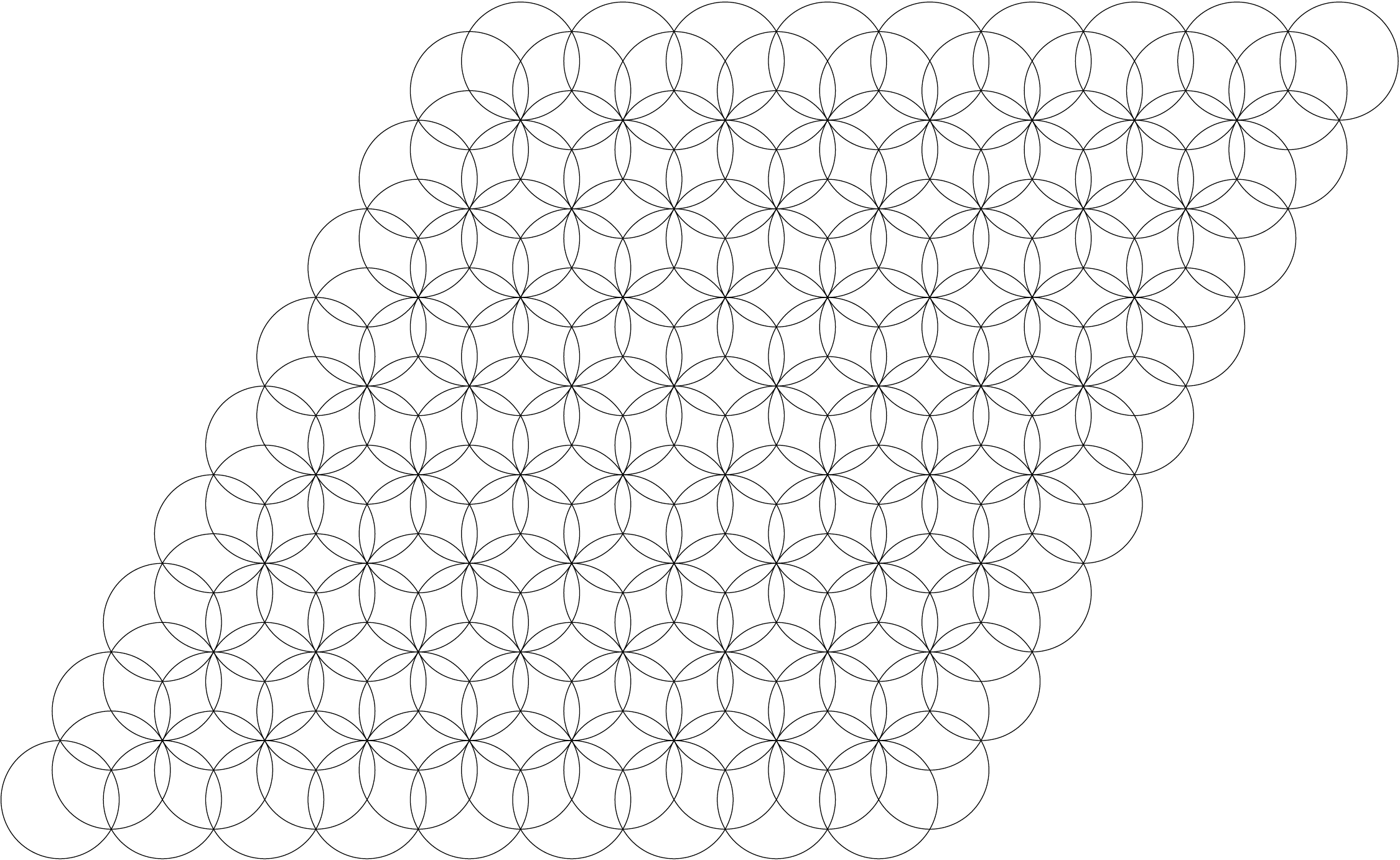}
	\end{minipage}
	\begin{minipage}{0.45\textwidth}
		\includegraphics[width=0.9\textwidth]{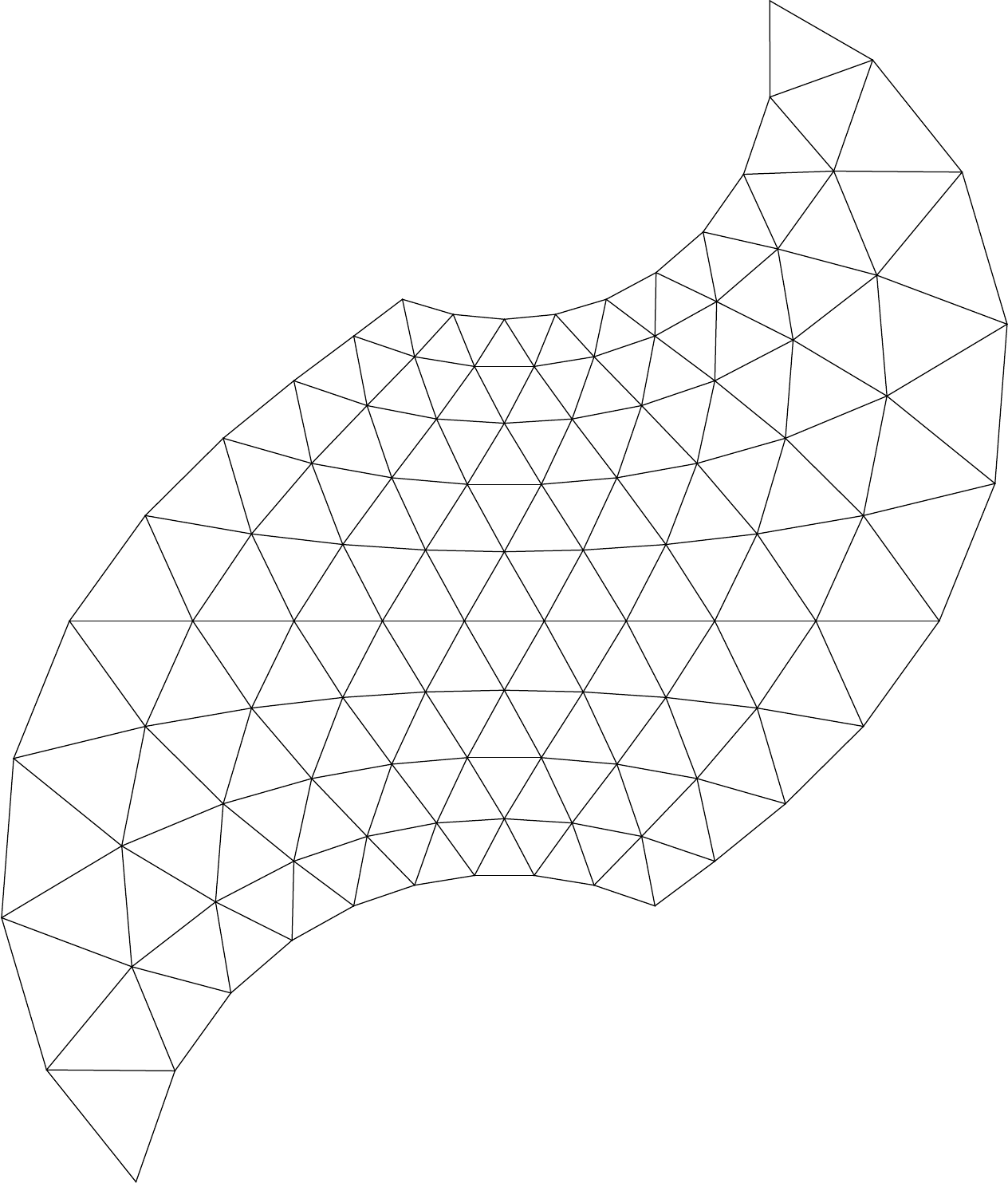}
	\end{minipage}
	\begin{minipage}{0.45\textwidth}
		\includegraphics[width=1\textwidth]{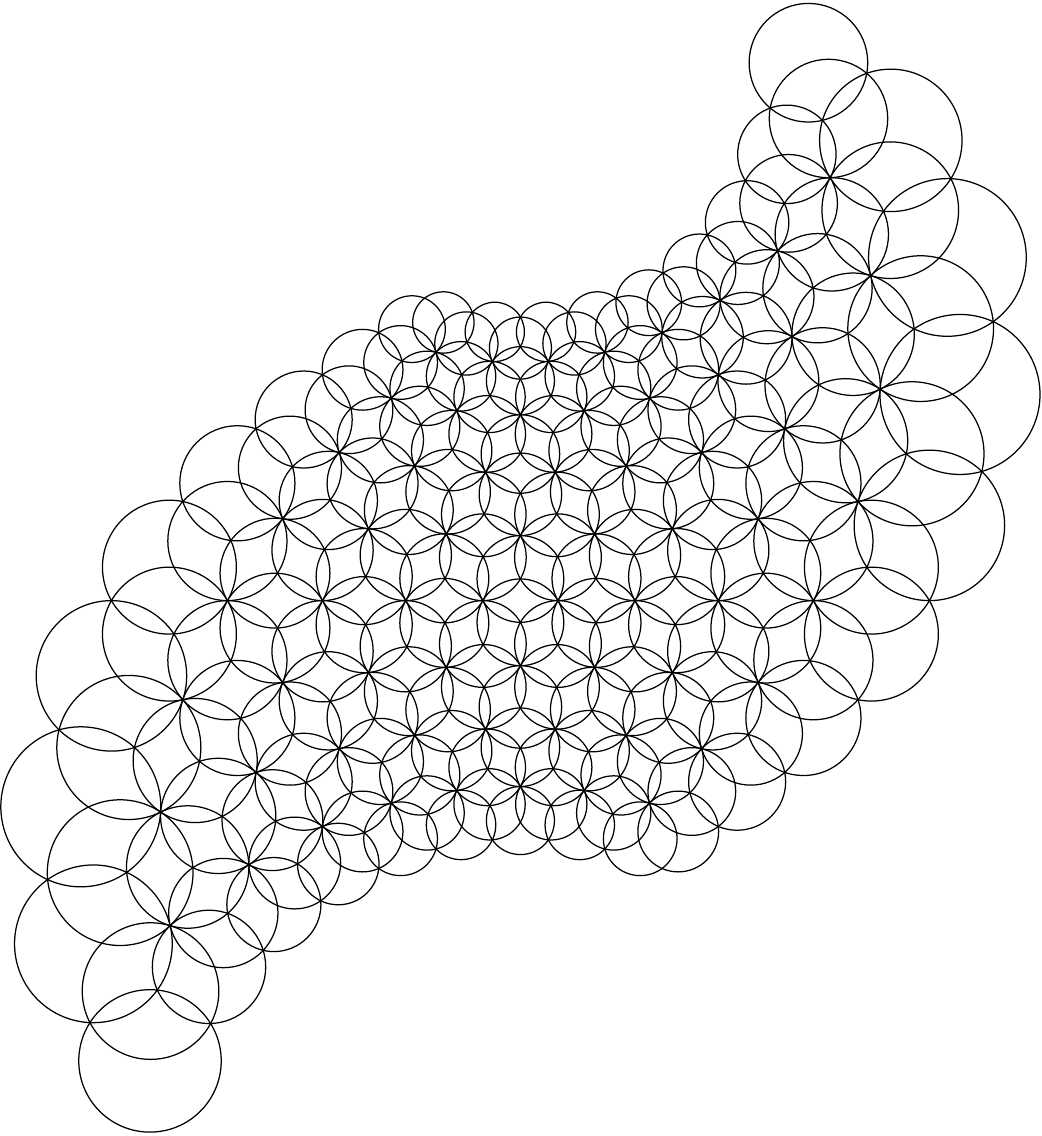}
	\end{minipage}
	\begin{minipage}{0.42\textwidth}
		\includegraphics[width=0.9\textwidth]{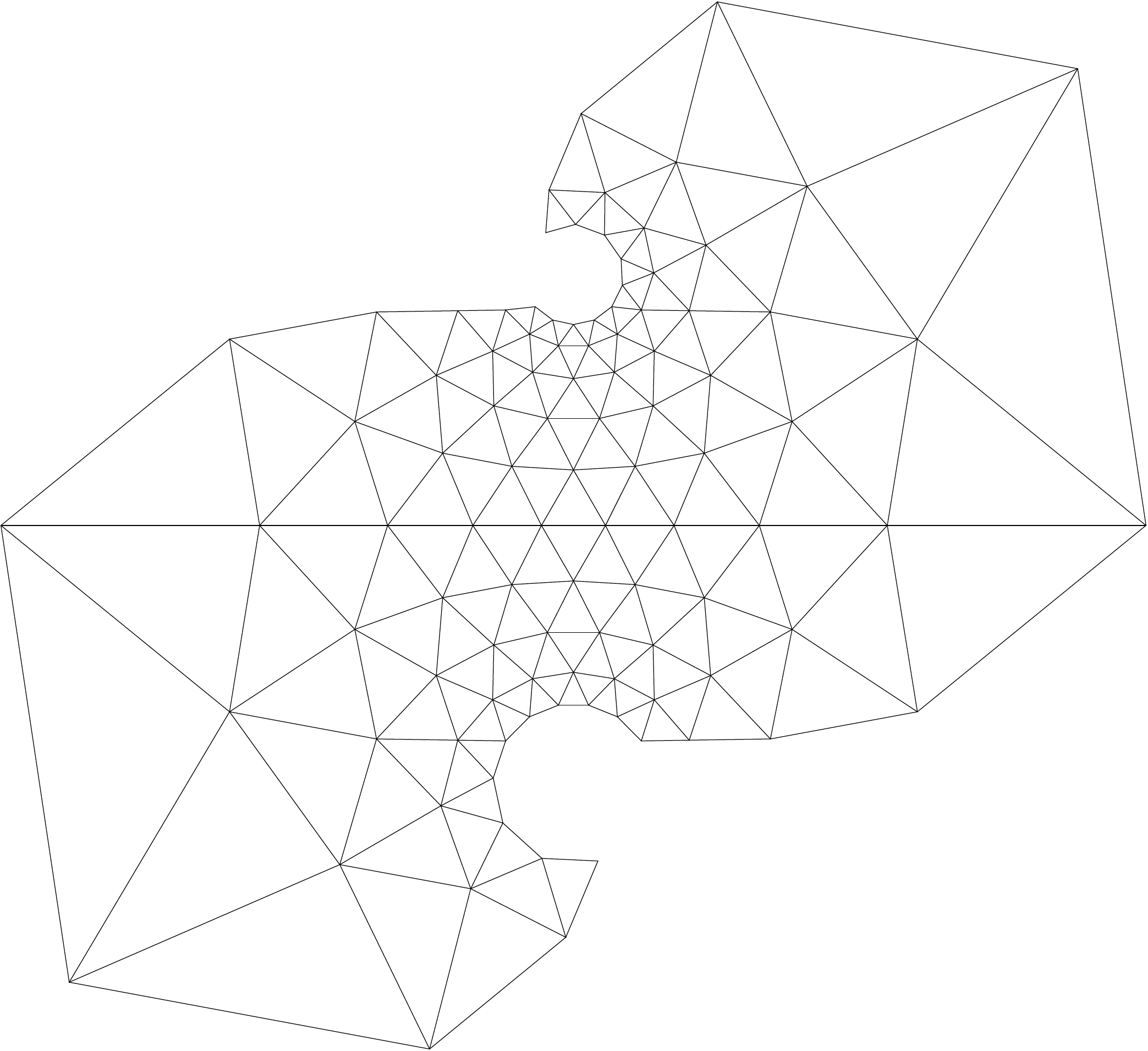}
	\end{minipage}
	\begin{minipage}{0.49\textwidth}
		\includegraphics[width=1\textwidth]{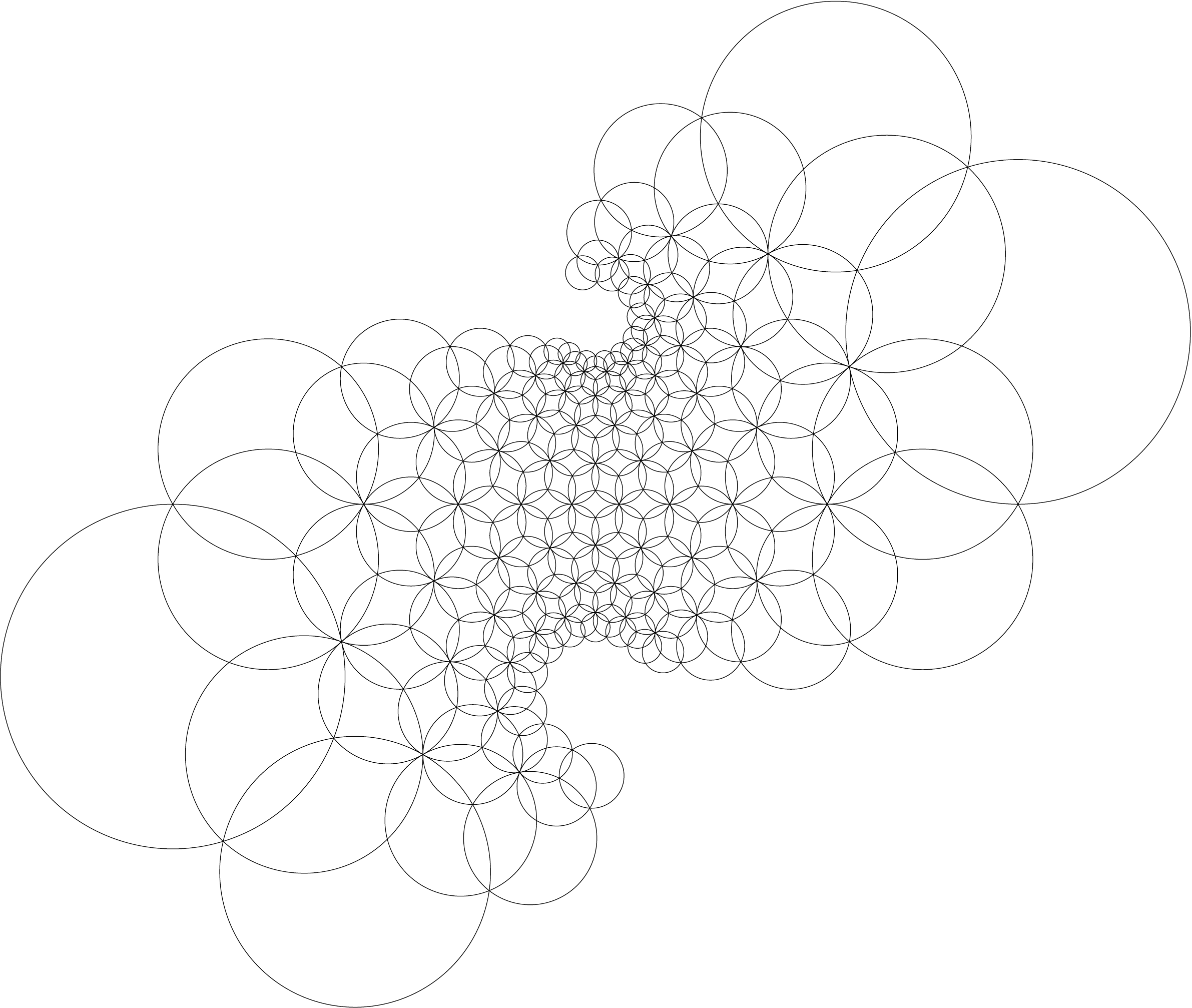}
	\end{minipage}
	\caption{Circle patterns with Delaunay intersection angle $\Theta \equiv \pi/3$ on a triangulated torus. Vertex positions $z$ and combinatorics are shown on the left. Their circumcircles are shown on the right. Notice that not all intersection points are the vertices under consideration. The three tori have different conformal structures. } 
	\label{fig:circlepatterns}
\end{figure}

\section{Background} \label{sec:background}

\subsection{M\"{o}bius transformations} \label{sec:mobius}

Without further notice, we only consider orientation-preserving M\"{o}bius transformations on the Riemann sphere. They are in the form $z \mapsto \frac{a z + b}{c z +d}$ for some $a,b,c,d \in \mathbb{C}$ such that $ad - bc =1$ and are also called complex projective transformations. These transformation are generated by Euclidean motions and inversion $z \mapsto 1/z$. In particular, they are holomorphic, map circles to circles and preserve cross ratios.

\subsection{Developing map and holonomy} \label{sec:develop}

\begin{proposition}
	Given a cross ratio system $\cratio:E \to \mathbb{C}$ on a triangulated surface $M$, we denote $\hat{M} = (\hat{V},\hat{E},\hat{F})$ the universal cover of $M$ with the pull back triangulation and $p$ the covering map. Then there exists a \emph{developing map} $z:\hat{V} \to \mathbb{C}\cup \{ \infty \}$ such that for every edge $ij \in \hat{E}$ shared by triangles $\{ijk\}$ and $\{jil\}$
	\begin{equation}\label{eq:developingmap}
	\cratio_{p(ij)} = -\frac{(z_k - z_i)(z_l -z_j)}{(z_i - z_l)(z_j - z_k)}.
	\end{equation}
	The developing map is unique up to complex projective transformations and has an equivariance property with respect to the fundamental group $\pi_1(M)$: For any $\gamma \in \pi_1(M)$, there exists a complex projective transformation $\rho_{\gamma} \in PSL_2(\mathbb{C})$ such that
	\[
	z \circ \gamma = \rho_{\gamma} \circ z.
	\]
	The map $\gamma \mapsto \rho_{\gamma}$ defines a holonomy representation $\hol$ of the cross ratio system unique up to conjugation by elements in $PSL_2(\mathbb{C})$, i.e. $\hol(\cratio) \in \mathcal{X}(M)$ represents a point in the character variety.
	
	Conversely, every mapping into the the extended complex plane with the equivariance property induces a cross ratio system on $M$ via \eqref{eq:developingmap}.
\end{proposition}
\begin{proof}

	To obtain a developing map, we pick a face $\{ijk\}$ of $\hat{M}$ and assign the vertices to three distinct points $z_i,z_j,z_k$. If $\{jil\}$ is a neighboring face (see Fig. \ref{fig:delaunay} left), then $z_l$ is uniquely determined by $ \cratio_{p(ij)} = -\frac{(z_k - z_i)(z_l -z_j)}{(z_j - z_k)(z_i - z_l)}$. In this way, nearby triangles are laid out one by one. However, for the developing to be well defined, we have to make sure such construction is independent of the order of the triangles chosen. Indeed by Lemma \ref{lem:cross2} below, the conditions in Definition \ref{def:crsys} yield that the holonomy of the developing map is trivial around each vertex. 
	
	The developing map $z$ is then uniquely determined as long as some triangle $z_i,z_j,z_k$ is prescribed. If the three vertices take some other distinct values $\tilde{z}_i,\tilde{z}_j, \tilde{z}_k$, then there exists a unique complex projective transformation $\rho$ mapping $z_i,z_j,z_k$ to $\tilde{z}_i,\tilde{z}_j, \tilde{z}_k$ and the developing maps are related by $\tilde{z} = \rho \circ z$.
	
	Conversely, Lemma \ref{lem:cross1} implies that a mapping $z:\hat{V} \to \mathbb{C}\cup \{\infty\}$ with the equivariance property induces a cross ratio system on $M$.
\end{proof}

\begin{lemma} \label{lem:cross1} Given a point $z_i$ and a sequence of numbers $z_0, z_1, ...., z_n=z_0, z_{n+1}=z_1$ in $\C\cup \{\infty\}$ such that $z_i, z_j, z_{j-1}, z_{j+1}$ are distinct for all $j$, we define the cross ratios for $j=1,\dots,n$ by
	\begin{equation*}
	 \cratio_{ij} = -\frac{(z_i-z_{j-1})(z_{j}-z_{j+1})}{(z_{j+1}-z_i)(z_{j-1}-z_j)}
	\end{equation*}
Then 
	\[ \prod_{j=1}^n \cratio_{ij} =1, \quad \sum_{k=1}^n  \prod_{j=1}^k \cratio_{ij} =0.\]
\end{lemma}
	\begin{proof}
	Under the above assumption, for $k=1,2,\dots,n$
	
	\begin{equation*}\prod_{j=1}^k \cratio_{ij} =\frac{(z_{0}-z_i)(z_1-z_i)}{z_0-z_1} \cdot \frac{z_k-z_{k+1}}{(z_k-z_i)(z_{k+1}-z_i)}= \frac{(z_{0}-z_i)(z_1-z_i)}{z_0-z_1}  (\frac{1}{z_k-z_i} - \frac{1}{z_{k+1}-z_i}) \end{equation*}
	
	\begin{equation*} 
	\sum_{k=1}^n \prod_{j=1}^k \cratio_{ij} =\frac{(z_{0}-z_i)(z_1-z_i)}{z_0-z_1}  (\frac{1}{z_1-z_i} - \frac{1}{z_{n+1}-z_i}) \end{equation*}
	In the special case that  $z_0=z_n$ and $z_1=z_{n+1}$, we obtain the claim.
\end{proof}

\begin{lemma}\label{lem:cross2}Suppose $\cratio_{i1}, ..., \cratio_{in} \in \C$ such that $$\prod_{j=1}^n \cratio_{ij}= 1, \quad \text{and} \quad \sum_{k=1}^n \prod_{j=1}^k \cratio_{ij} =0.$$
	Then there exist $z_i, z_0, z_1, ..., z_n =z_0, z_{n+1}=z_1$  $\in \C \cup \{\infty\}$  such that 
\[
	\cratio_{ij} = -\frac{(z_i-z_{j-1})(z_{j}-z_{j+1})}{(z_{j+1}-z_i)(z_{j-1}-z_j)}
	\]
	 for all $j=1,2,\dots,n$. The vector $(z_i, z_0, z_1, ..., z_n)$ is unique up to $PSL(2, \C)$ action.
\end{lemma}

\begin{proof} Once we show that there is a unique configuration with $z_i=\infty$, $z_1=1$ and $z_0=0$, the uniqueness follows since cross ratios are M\"{o}bius invariant and the fact that there always exists a unique M\"{o}bius transformations mapping any three distinct points to any three distinct points.
	
	Taking $z_i=\infty$, $z_1=1$ and $z_0=0$, we have
	\[
	\cratio_{ij}=\frac{z_{j+1}-z_j}{z_j- z_{j-1}}
	\]
	Hence using $z_0, z_1$, we can define inductively that
	$$ z_{k+1}=z_k+\prod_{j=1}^k \cratio_{ij}.$$
	We claim that $z_{n+1}=1=z_1$ and $z_n=0=z_0$. Indeed, consider the summation,
	$$\sum_{k=1}^n  z_{k+1}=\sum_{k=1}^n (z_k+\prod_{j=1}^k \cratio_{ij}).$$  We obtain that $z_{n+1}=z_1$ by the assumption that
	$\sum_{k=1}^n \prod_{j=1}^k (\cratio_{ij})=0$.
	Now consider $1=z_{n+1}=z_n+\prod_{j=1}^n (\cratio_{ij})= z_n+1$. It follows that $z_n=0=z_0$. Therefore the claim holds.
\end{proof}

\subsection{Delaunay condition} \label{sec:delaunay}

Since the surface $M$ is assumed to be oriented, every triangular face of $M$ is equipped with an orientation. Using this orientation together with a realization of the vertices on the Riemann sphere, we associate a circumdisk to every face.
\begin{definition}
	We write $\{ijk\}$ a  face with vertices $i,j,k$ that is oriented,  i.e. $\{ijk\}=\{jki\}=\{kij\}$ but $\{ijk\}\neq\{jik\}$. Suppose $z_i,z_j,z_k\in \mathbb{C}\cup\{\infty\}$ are distinct. Then there is an oriented circle $C_{ijk}$ passing through $z_i,z_j,z_k$ in cyclic order. It bounds two disks on the Riemann sphere. We denote $D_{ijk}$ the open disk with $C_{ijk}$ as the boundary in positive orientation and call $D_{ijk}$ the circumdisk of $\{ijk\}$ under $z$.
\end{definition}
\begin{figure}[h!] \centering
	\begin{minipage}{0.46\textwidth}
		\includegraphics[width=1\textwidth]{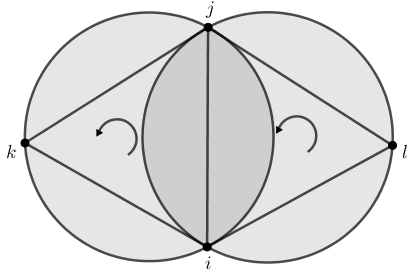}
	\end{minipage}
	\begin{minipage}{0.45\textwidth}
		\includegraphics[width=0.85\textwidth]{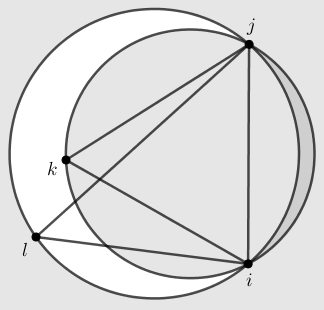}
	\end{minipage}
	\caption{Two possible configurations for Delaunay cross ratio $\Arg \cratio_{ij} \in [0,\pi)$. Left: Two bounded disks. Right: One bounded disk and one unbounded disk. Both figures satisfy the local empty circle condition.} 
	\label{fig:delaunay}
\end{figure}

\begin{proposition}	
	Suppose a cross ratio system is given and $z:V \to \mathbb{C}\cup\{\infty\}$ is a developing map of the universal cover. Then the cross ratio system satisfies $\Arg \cratio \in [0,\pi)$ if and only if the local empty circle condition is satisfied, i.e. for any two neighbouring faces $\{ijk\}$ and $\{jil\}$ (See Fig. \ref{fig:delaunay}), they satisfy
	\begin{enumerate}[(i)]
	   \item $z_k \notin D_{jil}$ and 
	   \item $z_l \notin  D_{ijk}$ and
	   \item $D_{jil} \cap D_{ijk} \neq \emptyset$
	\end{enumerate}

\end{proposition}
\begin{proof}
	Mapping $z_i$ to infinity by inversion, we have $\cratio_{ij} = \frac{z_k-z_j}{z_j-z_l}$ and the claim follows immediately. In particular, \[ D_{jil} \cap D_{ijk} = \emptyset\] if and only if $\Arg \cratio_{ij} = \pi$.
\end{proof}

\subsection{Complex projective structures} \label{sec:cp1}

Every Delaunay cross ratio system induces a complex projective structure on the closed surface together with a circle pattern. We recall the definition of a complex protective structure on a surface. 

\begin{definition}
	A complex projective structure on a surface $M$ is a maximal atlas of charts from open subsets of $M$ to the Riemann sphere such that the transition functions are restrictions of M\"{o}bius transformations. These charts are called projective charts.
	
	Two complex projective structures are marked isomorphic if there is a diffeomorphism homotopic to the identity mapping projective charts to projective charts. We denote $P(M)$ the space of marked complex projective structures.
\end{definition}

	Intuitively, a cross ratio system gives a recipe to glue circumdisks with transition functions as M\"{o}bius transformations. For every face $\tau=\{ijk\}$, one assigns the vertices to some distinct points $p_i,p_j,p_k$ in the plane to obtain a circumdisk $D_{\tau}$. Consider a neighboring face $\sigma=\{jil\}$, its circumdisk $D_{\sigma}$ has vertices at $q_j, q_i,q_l$. For the common edge $\{ij\}$, pick four points $z_i,z_j,z_k,z_l$ in the plane such that $\cratio_{ij}=-\frac{(z_k-z_i)(z_l-z_j)}{(z_j-z_k)(z_i-z_l)}$. Then there are unique M\"{o}bius transformations $\phi_{\tau,ij}$ mapping $p_i,p_j,p_k$ to $z_i,z_j,z_k$ and $\phi_{\sigma,ij}$ mapping $q_i,q_j,q_l$ to $z_i,z_j,z_l$. We define the transition map $\phi_{\tau}^{\sigma}:= \phi^{-1}_{\sigma,ij} \circ \phi_{\tau,ij} $ which is independent of the choice of $z$. Taking the cocycle condition into account, one can show that the disjoint union  $\sqcup D_{ijk}$ with a quotient relation $x \sim y$ if $y =\phi_{\tau}^{\sigma}(x)$ induces a surface with a complex projective structure. However, it should be careful that in general the resulting surface is not homeomorphic to $M$ unless the Delaunay condition is assumed (See section \ref{sec:one-vertex} for a counterexample). 
	
	\begin{proposition}\cite{Fock1993} \label{prop:cp1} 
		Every Delaunay cross ratio system defines a complex projective structure on $M$ together with a circle pattern. It yields a forgetful map $f: \mathcal{D} \to P(M)$ from the space of all Delaunay cross ratio systems $\mathcal{D}$ to the space of all marked complex projective structures $P(M)$.
	\end{proposition}

	The holonomy representation of a Delaunay cross ratio system is identical to that of the underlying complex projective structure.

\subsection{Discrete holomorphic quadratic differentials} \label{sec:hqd}
We introduce discrete holomorphic quadratic differentials as tangent vectors of $P(\Theta)$. We first rewrite the equations for cross ratio systems in Definition \ref{def:crsys}. 

\begin{definition}
	We define $\Phi:\mathbb{C}^{E} \to \mathbb{C}^{2V}$ as follows: for every vertex $i\in V$
	\begin{gather*}
	\Phi(\cratio)_{i,1} :=  \Pi_{j=1}^m \cratio_{ij} -1  \\
	\Phi(\cratio)_{i,2}:= \cratio_{i1} + \cratio_{i1} \cratio_{i2} + \cratio_{i1}\cratio_{i2}\cratio_{i3} + \dots +  \cratio_{i1}\cratio_{i2}\dots\cratio_{im}
	\end{gather*}
	Furthermore, given an angle structure $\Theta$, we define the exponential maps
	\begin{enumerate}[(i)]
	\item $\exp:\mathbb{C}^{E} \to \mathbb{C}^{E} $ by $\exp(X)_{ij}:= e^{X_{ij}}$ and
	\item  $\exp_{\Theta}: \mathbb{R}^{E} \to \mathbb{C}^{E} $ by $\exp_{\Theta}(X)_{ij}:= e^{X_{ij}+ \mathbf{i} \Theta_{ij}}$ where $\mathbf{i}= \sqrt{-1}$.
	\end{enumerate}
\end{definition}  
	Notice that $\exp_{\Theta}$ is diffeomorphic to its image and the space of cross ratio systems $P(\Theta)$ can be identified with the zero set $(\Phi\circ \exp_{\Theta})^{-1} \{0\}$. To show that it is a manifold, we need to show that the Jacobian $D(\Phi\circ \exp_{\Theta})$ has constant rank along the zero set. Equivalently its kernel has constant dimension.

 \begin{definition} \label{prop:jacobian}
	Suppose $X:E\to \mathbb{R}$ satisfies $\Phi \circ \exp_{\Theta}(X)=0$. We write $\cratio:=\exp_{\Theta}(X)$. Then $q:E\to \mathbb{R}$ is called a discrete holomorphic quadratic differential if it is in the kernel of the Jacobian $D(\Phi\circ \exp_{\Theta})_X$, i.e. for every vertex $i$
	\begin{align*}
	0=&D(\Phi\circ \exp_{\Theta})_X(q)_{i,1}= \sum_j q_{ij} \\
	0=&D(\Phi\circ \exp_{\Theta})_X(q)_{i,2} = q_{i1} \cratio_{i1} + (q_{i1} + q_{i2})\cratio_{i1} \cratio_{i2} +  \dots + (q_{i1}+ \dots + q_{im}) \cratio_{i1}\cratio_{i2}\dots\cratio_{im}
	\end{align*} 	
	We write the real vector space $Q^{\mathbb{R}}(\cratio):=\Ker(D(\Phi\circ \exp_{\Theta})_X)  \subset \mathbb{R}^{|E|}$. Similarly, we define the complex vector space $Q^{\mathbb{C}}(\cratio):=\Ker(D(\Phi\circ \exp)_X) \subset \mathbb{C}^{|E|}$.
\end{definition}

In the smooth theory, a holomorphic quadratic differential can be obtained via the Schwarzian derivative, which measures the degree how much a holomorphic map fails to be a M\"{o}bius transformation. Here we have a discrete analog: A discrete holomorphic quadratic differential $q$ is a tangent vector of $P(\Theta)$ that describes an infinitesimal deformation of the underlying circle pattern preserving the intersection angle $\Theta$. Such a deformation is induced by a M\"{o}bius transformation if and only if $q\equiv0$. 

\begin{proposition}\label{prop:lowerbound}
	For every cross ratio system $\cratio$ on a closed surface with genus $g$
	\begin{align*}
		\dim_{\mathbb{R}} Q^{\mathbb{R}}(\cratio)&\geq 6g-6 \\
		\dim_{\mathbb{C}} Q^{\mathbb{C}}(\cratio)&\geq 6g-6 + |V|
	\end{align*}
\end{proposition}
\begin{proof}
	It follows by counting the number of variables and linear constraints. An inequality is obtained since the constraints might be linearly dependent. 
	
	For $Q^{\mathbb{R}}(\cratio)$, there are $|E|$ real variables. For every vertex, there is one real equation and one complex equation. Hence \[\dim_{\mathbb{R}} Q^{\mathbb{R}}(\cratio)\geq |E|-|V|-2|V|=6g-6.\]
	
	For $Q^{\mathbb{C}}(\cratio)$, there are $|E|$ complex variables. For every vertex, there are two complex equations. Hence \[\dim_{\mathbb{C}} Q^{\mathbb{C}}(\cratio)\geq |E|-|V|-|V|=6g-6+|V|.\]
	
\end{proof}

By expressing the cross ratios in terms of the developing map $z$, we obtain the original formulation of discrete holomorphic quadratic differentials introduced in \cite{Lam2015a}.

\begin{proposition}
	Suppose a cross ratio system is given on $M=(V,E,F)$. We denote its developing map $z: \hat{V} \to \mathbb{C}\cup\{\infty\}$ of the universal cover $\hat{M}=(\hat{V},\hat{E},\hat{F})$. Then $q:E \to \mathbb{R}$ is a holomorphic quadratic differential if and only if for every vertex $i \in \hat{V}$
	\begin{align}
	\sum_j \hat{q}_{ij} &=0 \label{qsum} \\
	\sum_j \hat{q}_{ij}/(z_j - z_i) &=0 \label{zsum}
	\end{align} 
	where $\hat{q}$ is the lift of $q$ to the universal cover $\hat{M}$ \[\
	\hat{q}_{ij} := q_{p(ij)}.
	\]
\end{proposition}

Discrete holomorphic quadratic differentials are related to discrete minimal surfaces and discrete integrable systems. See \cite{Lam2016,Lam2017,Lam2015} for more details.

\begin{remark}
	Here we explain the smooth counterparts of Equations \eqref{qsum} and \eqref{zsum}. In the smooth theory, a holomorphic quadratic differential differential $q$ on a Riemann surface locally is of the form $q= f(z) \, dz^2$
	where $f$ is a holomorphic function. 	
	Equation \eqref{qsum} mimics the property in the smooth theory that
	\[
	q(X)+q(JX)+q(J^2 X) + q(J^3 X) =0
	\]
	where $X$ is any tangent vector and $J$ is the counterclockwise rotation of $\pi/2$. Particularly $q(X)=q(-X)$. The holomorphicity is encoded in Equation \eqref{zsum}: $f$ being holomorphic is equivalent to the 1-form
	\[
	q/dz = f(z) \, dz
	\]
	being closed, which holds if and only if $\oint q/dz=0$ over any contractible loop. 
\end{remark}

\subsection{Discrete harmonic functions: cotangent Laplacian} \label{sec:harmonic}

We recall some facts about a particular graph Laplacian, whose edge weights depend on a realization of the graph into the plane. It arises in the finite element discretization \cite{Pinkall1993} and has been applied to statistical mechanics \cite{Smirnov2010}. We will use it to study the Jacobian of $\Phi\circ \exp_{\Theta}$ and the projection to the Teichm\"{u}ller space in the next two sections.

\begin{definition}\label{def:cotlap}
	Suppose $z:V \to \mathbb{C}$ is a realization of a triangulated surface. Then a function $u:V \to \mathbb{R}$ is harmonic in the sense of the cotangent Laplacian if for every vertex $i$
	\[
	\sum_j c_{ij} (u_j - u_i) = 0
	\]
	where $c_{ij}=c_{ji}:=\cot \angle jki + \cot \angle ilj $ is called the cotangent weight (See Figure \ref{fig:delaunay} left). Here $\angle jki \in (-\pi,\pi)$ and it is positive if $z_j,z_k,z_i$ are in counterclockwise order and is negative otherwise. 
\end{definition}

There is a conjugate harmonic function for every discrete harmonic function.

\begin{proposition}
	On a simply connected domain, a function $u:V \to \mathbb{R}$ is harmonic if and only if there exists $u^*:F \to \mathbb{R}$ unique up to an additive constant such that
	\[
	u^*_{ijk} - u^*_{jil} = \frac{1}{2}(\cot \angle jki + \cot \angle ilj) (u_j-u_i)
	\]
	where $\{ijk\}, \{jil\}$ are the left and the right faces of the oriented edge pointing from $i$ to $j$.
\end{proposition}

As in the smooth theory, linear functions are discrete harmonic. For any two complex numbers $a,b$, we write their inner product $
\langle a,b \rangle = \Re( a \bar{b})$.

\begin{proposition}
	Given a realization $z:V \to \mathbb{C}$ of a triangulated surface, a function $u:V \to \mathbb{R}$ is called linear if there exists $a \in \mathbb{C}, b \in \mathbb{R}$ such that
	\[
	u_i = \langle a,z_i \rangle + b.
	\]
	It is harmonic with respect to the cotangent Laplacian and its conjugate harmonic function $u^*:F \to \mathbb{R}$ is of the form
	\begin{equation}\label{eq:charmonic}
			u^*_{ijk} = \langle a, -\mathbf{i} z^*_{ijk} \rangle + d
	\end{equation}
	for some $d \in \mathbb{R}$ where $z^*_{ijk}$ is the circumcenter of triangle $\{ijk\}$ and $\mathbf{i}= \sqrt{-1}$.
\end{proposition}
\begin{proof}
	Notice that for any two neighboring triangles $\{ijk\},\{jil\}$, the difference between the circumcenters can be written as
	\[
	z^*_{ijk} - z^*_{jil} = \frac{\mathbf{i}}{2}(\cot \angle jki + \cot \angle ilj) (z_j -z_i).
	\]
	Suppose $u$ is linear and of the form $u= \langle a,z \rangle + b$, we define $u^*:F \to \mathbb{R}$ via $\eqref{eq:charmonic}$. Then
	\begin{align*}
			u^*_{ijk} - 	u^*_{jil} =  \langle a, -\mathbf{i}  z^*_{ijk} \rangle - \langle a, -\mathbf{i} z^*_{jil} \rangle &= \frac{1}{2}(\cot \angle jki + \cot \angle ilj) ( \langle a,z_j\rangle -\langle a, z_i \rangle )\\ &=\frac{1}{2}(\cot \angle jki + \cot \angle ilj) (u_j-u_i)
	\end{align*}
Hence $u$ is harmonic with $u^*$ as its conjugate harmonic function.
\end{proof}

It is known that there is a correspondence between discrete holomorphic quadratic differentials and discrete harmonic functions with respect to the cotangent Laplacian. We outline here the results from \cite{Lam2015a}. 

\begin{proposition}[\cite{Lam2015a}] \label{prop:hqdharmonic}
	Suppose $z:V \to \mathbb{C}$ is a realization of a simply connected triangular mesh. Then there is a one-to-one correspondence between the following:
	\begin{enumerate}[(i)]
		\item a discrete holomorphic quadratic differential $q:E \to \mathbb{R}$;
		\item an infinitesimal deformation of the developing map $\dot{z}:V \to \mathbb{C}$ such that the argument of the cross ratios  $\Arg \cratio=\Imaginary \log \cratio$ is preserved. This $\dot{z}$ is unique up to an infinitesimal M\"{o}bius transformation;
		\item  a discrete harmonic function $u:V\to \mathbb{R}$ unique up to a linear function.
	\end{enumerate}
\end{proposition}
	These correspondences have been proved in \cite{Lam2015a}. Since we shall use them in Section \ref{sec:upperbound} and \ref{sec:immersion}, we sketch the relations here.
	
	$(1)\leftrightarrow(2)$: Given a first order deformation $\dot{z}:V \to \mathbb{C}$, we compute the logarithmic change in the cross ratio (see Fig \ref{fig:delaunay} left for the notation)
	\[
	q_{ij}:= \frac{\dot{z}_i-\dot{z}_k}{z_i-z_k}-\frac{\dot{z}_l-\dot{z}_i}{z_l-z_i}+\frac{\dot{z}_j-\dot{z}_l}{z_j-z_l}-\frac{\dot{z}_k-\dot{z}_j}{z_k-z_j} \left( =\frac{\dot{\cratio}_{ij}}{\cratio_{ij}}  \right)
	\]
	which yields a holomorphic quadratic differential. Conversely $\dot{z}$ is determined by $q$ whenever $\dot{z}_i,\dot{z}_j,\dot{z}_k$ is prescribed for some face $\{ijk\}$. Here $q$ is real-valued because $\dot{z}$ preserves $\Imaginary \log \cratio$.
	
	$(2)\leftrightarrow (3)$: $\dot{z}$ preserve $\Imaginary \log \cratio$ if and only if there exists a function $u:V\to\mathbb{R}$ such that for every edge $\{ij\}$
	\[
	\Imaginary \frac{\dot{z}_j-\dot{z}_i}{z_j-z_i} = \frac{u_i + u_j}{2}.
	\]
	The function $u$ is uniquely determined from $\dot{z}$ as follows: Pick a face $\{ijk\}$ then
	\[
	u_i := 	\Imaginary( \frac{\dot{z}_j-\dot{z}_i}{z_j-z_i}+\frac{\dot{z}_i-\dot{z}_k}{z_i-z_k}- \frac{\dot{z}_k-\dot{z}_j}{z_k-z_j})
	\] 
	is independent of the choice of faces containing vertex $i$. It turns out that $u$ is a discrete harmonic function with respect to the cotangent Laplacian. It is a linear function if and only if $\dot{z}$ is induced by an infinitesimal M\"{o}bius transformation. The conjugate harmonic function $u^*:F \to \mathbb{R}$ similarly satisfies
	\begin{align} \label{eq:conjharmonic}
		\Real \frac{\dot{z}_j-\dot{z}_i}{z_j-z_i} = -\frac{u^*_{ijk} \cot \angle ilj  + u^*_{ilj} \cot \angle jki }{\cot \angle ilj  + \cot \angle jki }.
	\end{align}

\section{Circle patterns on complex projective tori} \label{sec:cp1tori}

\subsection{Complex affine torus and holonomy} \label{subsec:holonomy}

For the torus, it is known that every complex projective structure can be reduced to an affine structure. 

\begin{definition}
	A complex affine structure on a surface $M$ is a maximal atlas of charts from open subsets of $M$ to $\mathbb{C}$ such that the transition functions are restrictions of complex affine transformations $z \mapsto a z + b$ for some $a,b \in \mathbb{C}$ with $a \neq 0$.
\end{definition}

It is elementary to characterize the holonomy representation of the fundamental group of the torus in $PSL(2,\mathbb{C})$ up to conjugation. Assume $\gamma_1,\gamma_2$ are generators of the fundamental group of the torus, which satisfy $\gamma_1 \gamma_2 = \gamma_2 \gamma_1$. Pick a developing map $z: \hat{M} \to \mathbb{C} \cup \{\infty\}$ for the complex projective structure. The corresponding holonomy $\rho_1,\rho_2$ in $PGL(2,\mathbb{C})$ satisfy
$\rho_1 \rho_2 =  \rho_2 \rho_1$. If $\rho_1$ is not the identity, there are three types:
\begin{enumerate}[(I)]
	\item $\rho_1$ has only one fixed point. Then it is also the only fixed point of $\rho_2$. We can normalize the developing map by composing a M\"{o}bius transformation such that the fixed point is at infinity and the holonomy $\rho_1,\rho_2$ are translation, i.e. $\exists \beta_r \in \mathbb{C}$ such that \[
	(z\circ \gamma_r)_i = \rho_r (z_i) = z_i + \beta_r  \quad \forall \, i \in \hat{V}.\]
	\item $\rho_1$ has two fixed points and $\rho_2$ does not exchange them. Then they are fixed points of $\rho_2$ as well. We can normalize the developing map such that the fixed points are at the origin and infinity. The holonomy becomes stretched rotation, i.e. $\exists \alpha_r \in \mathbb{C}$ such that \[
	(z\circ \gamma_r)_i = \rho_r (z_i) = \alpha_r z_i  \quad \forall \, i \in \hat{V}.\]
	\item $\rho_1$ has two fixed points and $\rho_2$ does exchange them. We can normalize the developing map such that the fixed points of $\rho_1$ are the origin and infinity. Then the holonomy is of the form
	\[
	\rho_1 (z_i) = - z_i  \quad  \rho_2 (z_i) = \alpha_2 /z_i  \quad \forall \, i \in \hat{V}.\]
\end{enumerate}
In the case where $\rho_1$ is the identity, then we either have case (I) or (II) depending on the number of fixed points of $\rho_2$. Notice that the holonomy in case (I) and (II) becomes affine transformations.

Complex projective structures on the torus are induced by affine structures (See \cite[Ch.9, p.189-192]{Gunning1966} and \cite{Loray2009}). More details can be found in the survey \cite{Baues2014}.

\begin{proposition}[Gunning]\label{prop:gunning}
	Every complex projective structure on a torus can be reduced to an affine structure.
\end{proposition}

\begin{corollary}\label{cor:affinedev}
	Every developing map of a Delaunay cross ratio system on a torus can be normalized to have holonomy as complex affine maps (type I or II). In particular, all circumdisks are bounded and the Euclidean image of any two neighboring triangles are locally embedded, i.e. there is no fold.
\end{corollary}
\begin{proof}
	Every Delaunay cross ratio system induces a complex projective structure on the torus and thus an affine structure by Proposition \ref{prop:gunning}. The developing map for the affine structure has image in $\mathbb{C}$. Hence no circumdisk under this developing map passes through infinity and thus every circumdisk is bounded. It implies the Euclidean images of any two neighboring triangles have no fold, i.e. Figure \ref{fig:delaunay} (right) does not occur.
\end{proof}

	The holonomy representation of type III is not liftable to $SL(2,\mathbb{C})$ and hence excluded from complex projective tori. However it does occur for some non-Delaunay cross ratio systems. See \ref{sec:one-vertex} for examples.

\subsection{Delaunay cross ratio systems on the torus} \label{sec:delcrstori}

Making use of the affine structures, in this subsection we aim to prove that

\begin{theorem}\label{prop:kernel}
	The space $Q^{\mathbb{R}}(\cratio)$ of discrete holomorphic quadratic differentials has dimension $2$ for any Delaunay cross ratio system $\cratio$ on the torus. 
\end{theorem}

\begin{proof}
	It follows from Lemma \ref{lem:g1lower} that $\dim_{\mathbb{R}} Q^{\mathbb{R}}(\cratio) \geq 2$ and Lemma \ref{lem:g1upper} that   $\dim_{\mathbb{R}} Q^{\mathbb{R}}(\cratio) \leq 2$.
\end{proof}

Before discussing the proofs of the lemmas, we mention here an immediate consequence:

\begin{corollary}\label{cor:realsurface}
	The space of cross ratio systems with prescribed Delaunay angle structure $\Theta$ is a real analytic surface in the algebraic variety $\Phi^{-1}\{0\}$.
\end{corollary}
\begin{proof}
	Recall that $Q^{\mathbb{R}}(\cratio)$ is precisely the kernel of the Jacobian $D(\Phi\circ \exp_{\Theta})$ at $\exp^{-1}(\cratio)$. Theorem \ref{prop:kernel} implies that the Jacobian $D(\Phi\circ \exp_{\Theta})$ has constant rank along $(\Phi\circ \exp_{\theta})^{-1}\{0\} \subset \mathbb{R}^{E}$. The constant rank theorem yields that $(\Phi\circ \exp_{\theta})^{-1}\{0\} \subset \mathbb{R}^{E} \subset \mathbb{C}^{E}$ is a real analytic surface. 
\end{proof}

The corollary might also be induced from Rivin's result in Section \ref{sec:Rivin}, though it requires special attention when considering Euclidean tori. In the following two subsections, we provide a proof involving discrete harmonic functions instead, which will be an essential tool to study the projection from the space of Delaunay cross ratio systems to the Teichm\"{u}ller space in Section \ref{sec:immersion}.

\subsubsection{Lower bound for $\dim_{\mathbb{R}} Q^{\mathbb{R}}(\cratio)$} \label{sec:lower}
Notice that the lower bound in Proposition \ref{prop:lowerbound} does not provide any information for $\dim_{\mathbb{R}} Q^{\mathbb{R}}(\cratio)$ since $6g-6=0$ for the torus ($g=1$).

\begin{lemma}\label{lem:g1lower}
	For a Delaunay cross ratio system $\cratio$ on the torus, we have $\dim_{\mathbb{R}} Q^{\mathbb{R}}(\cratio) \geq 2$.
\end{lemma}
\begin{proof}
	We show that there is linearly dependence in the equations for discrete holomorphic quadratic differentials using the affine developing map $z: \hat{V} \to \mathbb{C}$ in Corollary \ref{cor:affinedev}. There are two cases:  
	
	Case (1): the holonomy consists of translations (type I). Pick an arbitrary function $q:E \to \mathbb{R}$. Then the complex number $q_{ij}/(z_j-z_i)$ is well defined on the oriented edges $e_{ij}$ of $M$. Namely, we have
	\[
	q_{ij}/(\rho_k(z_j)-\rho_k(z_i)) = q_{ij}/(z_j-z_i)
	\]
	for $k=1,2$. Thus
	\[
	\sum_i \sum_j q_{ij}/(z_j-z_i) = 0
	\]
	where the sum on the right is over the neighboring vertices $j$ of $i$ on the universal cover and the sum on the left is over all the vertices  $i$ in a fundamental domain.
	
	Case (2): the holonomy consists of stretched rotations (type II). Pick an arbitrary $q:E\to \mathbb{R}$ satisfying $\sum_j q_{ij}=0$. The complex number $z_i q_{ij}/(z_j-z_i)$ is well defined on the oriented edges $e_{ij}$ of $M$. Namely, we have
	\[
	\rho_k(z_i) q_{ij}/(\rho_k(z_j)-\rho_k(z_i)) = z_i q_{ij}/(z_j-z_i)
	\]
	for $k=1,2$. Thus we always have
	\[
	\sum_i  \sum_j z_i \frac{q_{ij}}{(z_j-z_i)} = \frac{1}{2}\sum_i  \sum_j \frac{q_{ij} (z_i -z_j)}{(z_j-z_i)} = -\frac{1}{2}\sum_i  \sum_j q_{ij} =0
	\]
	where the sum on the right is over the neighboring vertices $j$ of $i$ on the universal cover and the sum on the left is over all the vertices  $i$ in a fundamental domain.
	
	In both cases, the complex constraints
	\[
	\sum_j q_{ij}/(z_j -z_i) =0 \quad \forall i \in V
	\]
are linearly dependent. Hence
	\[
	\dim_{\mathbb{R}} Q^{\mathbb{R}}(\cratio) \geq |E| - 3|V| + 2 = 2.
	\]
\end{proof}

\subsubsection{Upper bound for $\dim_{\mathbb{R}} Q^{\mathbb{R}}(\cratio)$} \label{sec:upperbound}

We make use of the maximum principle for the cotangent Laplacian and the correspondence in Proposition \ref{prop:hqdharmonic} in order to obtain an upper bound for $\dim_{\mathbb{R}} Q^{\mathbb{R}}(\cratio)$.

\begin{proposition}
	Suppose $z:\tilde{V}\to \mathbb{C}$ is a developing map with affine holonomy induced from a Delaunay cross ratio system on a triangulated torus $M=(V,E,F)$. Then the cotangent Laplacian has the following property:
	\begin{enumerate}[(i)]
		\item the cotangent weights are invariant under deck transformations, i.e. $c_{\gamma(ij)}=c_{ij}$	for any deck transformation $\gamma$ on the universal cover;
		 \item the cotangent weights are non-negative;
		 \item the maximum principle holds: a discrete harmonic function $u:\tilde{V} \to \mathbb{R}$ achieving a local minimum or maximum at an interior vertex must be constant. 
	\end{enumerate}
\end{proposition}
\begin{proof}
	Firstly, the cotangent weights depend on the angles of triangles which are preserved by the affine holonomy. Hence the weights are invariant under deck transformations.
	
	Secondly, by Corollary \ref{cor:affinedev}, the circumdisks are bounded and hence all faces under $z$ are counterclockwisely oriented. By Definition \ref{def:cotlap}, all the angles within faces take positive values and
	\[
	c_{ij} = \cot \angle jki + \cot \angle ilj = \frac{\sin(\angle jki  + \angle ilj)}{\cos \angle jki  \, \cos  \angle jki} =  \frac{\sin(\Arg \cratio_{ij}) }{\cos \angle jki  \, \cos  \angle jki} \geq 0
	\]
	
	Thirdly, since the cotangent weights are non-negative, the proof for the maximum principle is a standard argument for graph Laplacian: Suppose $u$ is a discrete harmonic function with a local minimum at an interior vertex $i$ then
	\[
	0= \sum_j c_{ij} (u_j-u_i) \geq 0
	\]
	The inequality is strict unless $u_j = u_i$ for all $\{ij\} \in E$ with $c_{ij}\neq0$. Because the graph is connected, it implies $u$ is constant.
\end{proof}

Under an affine developing map, we can characterize discrete harmonic functions on the universal cover that correspond to discrete holomorphic quadratic differentials on the torus ascertained by Proposition \ref{prop:hqdharmonic}. 

\begin{lemma}\label{lem:periodicu}
	Suppose $z$ is a developing map with affine holonomy by $\rho_r = \alpha_r z_j + \beta_r$
	for $\alpha_r, \beta_r \in \mathbb{C}$ and $r=1,2$. Let $q:\hat{E} \to \mathbb{R}$ be a lift of a holomorphic quadratic differential from $M$, which satisfies $ (q\circ \gamma)_{ij} = q_{ij}$ and $u$ be one of its corresponding harmonic functions. Then there exists $a_r \in \mathbb{C}, b_r \in \mathbb{R}$ such that
	\[
	(u \circ \gamma_r)_i - u_i = \langle a_r, z_i \rangle + b_r
	\]
	with 	\begin{align*}
	a_1 \bar{\beta}_2 &= a_2 \bar{\beta}_1 \\
	a_2 (\bar{\alpha}_1 - 1) &= a_1  (\bar{\alpha}_2 - 1) .
	\end{align*}
\end{lemma}
\begin{proof}
	We define $v_{r,i}:= 	(u \circ \gamma_r)_i - u_i $ for every vertex $i$ and $r=1,2$. Then $v_1,v_2$ are discrete harmonic functions again since the cotangent weights are invariant under deck transformations. In addition the corresponding holomorphic quadratic differential of $v_1,v_2$ are identically zero because $q\circ\gamma =q$. Hence $v_1,v_2$ are linear functions, i.e. there exists $a_r \in \mathbb{C}$ and $b_r \in \mathbb{R}$ such that for every $i$
	\[
	v_{r,i} = (u \circ \gamma_r)_i - u_i=  \langle a_r, z_i \rangle + b_r.
	\]
	Since $\rho_1,\rho_2$ commute, we have
	\[
	(u\circ \gamma_2 \circ \gamma_1) - u = (u\circ \gamma_1 \circ \gamma_2)- u
	\]
	which implies for all $i$
	\[
	\langle a_2, \alpha_1 z_i + \beta_1 \rangle + b_2 + \langle a_1,z_i \rangle + b_1 = 	\langle a_1, \alpha_2 z_i + \beta_2 \rangle + b_1 + \langle a_2,z_i \rangle + b_2
	\]
	and thus
	\begin{align*}
	\langle a_2, \beta_1 \rangle = \langle a_1, \beta_2 \rangle &\implies \Re ( a_1 \bar{\beta}_2) = \Re (a_2 \bar{\beta}_1 )\\
	a_2 (\bar{\alpha}_1 - 1) &= a_1  (\bar{\alpha}_2 - 1)
	\end{align*}
	We consider the conjugate harmonic function $v^*_r$ of $v_r$. They are of the form
	\begin{align*}
	v^*_{r,klm} = \langle a_r, -\mathbf{i} z^*_{klm} \rangle + \tilde{b}_r .
	\end{align*}
	for every face $\{klm\}$. Applying a similar argument to $v^*_j$, we can deduce
	\[
	\Im ( a_1 \bar{\beta}_2) = \Im (a_2 \bar{\beta}_1 )
	\]
	Thus $a_1 \bar{\beta}_2 = a_2 \bar{\beta}_1$.
\end{proof}

We apply the above lemma to two cases: (i) the affine holonomy has only one fixed point at infinity where we have $\alpha_1=\alpha_2=1$ and (ii) the affine holonomy shares an additional fixed point at the origin where we have $\beta_1=\beta_2=0$.

\begin{lemma} \label{lem:dimU}
	Suppose the developing map $z$ is a Delaunay triangulation with the holonomy as affine transformations. We consider harmonic functions $u:\tilde{V}\to \mathbb{R}$ in two cases:
	\begin{enumerate}[(i)]
		\item $\rho_r(z)=z+\beta_r$ for some $\beta_r \in \mathbb{C}$. We define
		\[
		\mathcal{U}:=\{ u \text{ harmonic}| \exists a_r \in \mathbb{C}, b_r \in \mathbb{R} \text{ s.t.} (u \circ \gamma_r)_i - u_i = \langle a_r, z_i \rangle + b_r \, \& \, a_1 \bar{\beta}_2 = a_2 \bar{\beta}_1 \}.
		\]
		\item $\rho_r(z)=\alpha_r  z$ for some $\alpha_j \in \mathbb{C}$. We define
		\[
		\mathcal{U}:=\{ u \text{ harmonic}| \exists a_r \in \mathbb{C}, b_r \in \mathbb{R} \text{ s.t.} (u \circ \gamma_r)_i - u_i = \langle a_r, z_i \rangle + b_r \, \& \, a_2 (\bar{\alpha}_1 - 1) = a_1  (\bar{\alpha}_2 - 1) \}.
		\]
	\end{enumerate}
	Then in both cases
	\[
	\dim \mathcal{U} \leq 5.
	\]
\end{lemma}
\begin{proof}
	Notice that $ \mathcal{U}$ is a real vector space. If $u \in \mathcal{U}$ with $a_1=a_2=b_1=b_2=0$, then $u$ is a bounded harmonic function and hence $u$ is constant as a result of the maximum principle. Furthermore, since $a_1,a_2 \in \mathbb{C}$ and $b_1,b_2 \in \mathbb{R}$ satisfy a non-trivial condition $a_1 \bar{\beta}_2 = a_2 \bar{\beta}_1$ for case (i) and $ a_2 (\bar{\alpha}_1 - 1) = a_1  (\bar{\alpha}_2 - 1)$ for case (ii), we conclude that  $\dim \mathcal{U} \leq 5$.
\end{proof}

\begin{lemma}\label{lem:g1upper}
	For a Delaunay cross ratio system $\cratio$ on a triangulated torus, we have
	\[
	\dim_{\mathbb{R}} Q^{\mathbb{R}}(\cratio) \leq 2
	\]
\end{lemma}
\begin{proof}
	The set  $\mathcal{U}$ contains a 3-dimensional subspace of linear functions which correspond to trivial discrete holomorphic quadratic differentials by Proposition \ref{lem:dimU}. It implies the space of holomorphic quadratic differentials is of dimension $\leq \dim \mathcal{U} -3=2$.
\end{proof}

\subsection{Rivin's variational approach} \label{sec:Rivin}

\begin{proposition}\cite[Theorem 7.2]{Rivin} \label{prop:rivin}
	Let $\Theta$ be a Delaunay angle structure on a triangulated torus and $A_1,A_2 \in \mathbb{R} $. Then there exists a unique affine structure on the torus with affine holonomy  $\rho_r(z)=\alpha_r z + \beta_r$ such that $\log |\alpha_r|= A_r$ and the induced cross ratio $\cratio$ satisfies $\Arg \cratio \equiv \Theta$. The map $z$ is unique up to a global affine transformation.
\end{proposition}

Two different affine structures might correspond to the same complex projective structure. The affine holonomy always have a common fixed point at $\infty$. As long as it shares another fixed point in $\mathbb{C}$, we can apply an inversion to interchange the two fixed points and obtain a new affine structure, while the underlying complex projective structure remains the same. In the notation of Proposition \ref{prop:rivin}, an affine holonomy shares two fixed points if and only if $(A_1,A_2)\neq0$. 

\begin{proof}[Proof of Theorem \ref{thm:holo}]
	Firstly, $P(\Theta)$ being a real analytic surface in the algebraic variety is asserted in Corollary \ref{cor:realsurface}. Proposition \ref{prop:rivin} implies that the affine tori with $\Arg \cratio \equiv \Theta$ are parameterized by two real numbers $(A_1,A_2)$. By uniqueness, two different affine tori $(A_1,A_2)$ and $(\tilde{A}_1,\tilde{A}_2)$ share the same complex projective structure if and only if $(\tilde{A}_1,\tilde{A}_2)=(-A_1,-A_2)$. Thus the space of affine tori with $\Arg \cratio \equiv \Theta$ is a 2-to-1 covering of $P(\Theta)$ branched at the Euclidean torus $(A_1,A_2)=(0,0)$. Since the space of affine tori with $\Arg \cratio \equiv \Theta$ is homeomorphic to $\mathbb{R}^2$ (parameterized by $(A_1,A_2)$), we deduce that  $P(\Theta)$ is homeomorphic to $\mathbb{R}^2$.
	
	Secondly, we claim that the holonomy representation is an embedding. Suppose two Delaunay cross ratio systems represent the same point in the character variety. Since Delaunay cross ratio systems induce complex projective structures and thus affine structures on the torus, their holonomy representation can be reduced to affine transformations. By interchanging the fixed points, the affine holonomy of the Delaunay cross ratio systems coincide. By the uniqueness in \ref{prop:rivin}, the two Delaunay cross ratio systems must be identical.
	
	Thirdly, if two Delaunay cross ratio systems induce the same complex projective structures, then they represent the same point in the character variety. Hence the cross ratio systems are the same again by the uniqueness.
\end{proof}

\section{Projection to the Teichm\"{u}ller space of the torus} \label{sec:projection}

Given a Delaunay cross ratio system on the torus, it induces a complex projective structure and thus a conformal structure. The underlying conformal structure can be read off easily from the affine developing map.

Recall that we can parameterize the affine structures with a fixed underlying conformal structure on the torus by a complex parameter $c$ as follows: Start with a Euclidean torus obtained by gluing the opposite sides of a parallelogram spanned by complex numbers $1$ and $\tau$ in the upper half plane. The parameter $\tau$ represent a marked conformal structure in the Teichmm\"{u}ller space. Let $d$ be a developing map of an affine structure with the same marked conformal structure. Its holonomy satisfies $d(z+1)= \alpha_1 d(z) + \beta_1$ and $d(z+\tau) = \alpha_2 d(z) + \beta_2$. Notice that $d$ is holomorphic and $d' \neq 0$. We have $d''/d'$ holomorphic and periodic on the torus. Thus $d''/d' = c$ for some constant $c\in \mathbb{C}$. In the case $c=0$, $d(z) =az + b$ and we get a Euclidean torus. In the case $c\neq 0$, we have $d(z) = a e^{c z} + b$ for some constants $a,b$, which can be normalized to $d(z)= e^{cz}$ by translation and scaling. Its developing map satisfies $d(z+1) = e^c d(z)$ and $d(z+\tau)= e^{c\tau} d(z)$. Thus the holonomy is generated by 
\begin{gather*}
z\circ \gamma_1 = \rho_1(z) = e^{c}z \\ z\circ \gamma_2= \rho_2(z) = e^{c \tau}z
\end{gather*}

Given an affine developing map $z:\tilde{V} \to \mathbb{C}$, we are going to find its underlying conformal structure. Notice that it is incorrect to conclude that $c$ and $c \tau$ are respectively $\log \frac{(z\circ \gamma_1)_i}{z_i}$ and $\log \frac{(z\circ \gamma_2)_i}{z_i}$ since the branch for the imaginary part of $\log$ is unclear. However, it could be fixed as follows: Associating  every face (i.e. dual vertex) with the circumcenter defines a map $z^{*}:F \to \mathbb{C}$ which has the same affine holonomy representation as $z$. Pick a simple path of the dual graph on the universal cover that projects to a loop homotopic to $\gamma_r$ on the torus. Under $z^*$, the path becomes a polygonal curve with vertices $(z^{*}_0,z^{*}_1,\dots,z^{*}_n=\rho_r(z^{*}_0),z^{*}_{n+1}=\rho_r(z^{*}_{1}))$ and we define
\begin{align}\label{eq:hhol}
	h_r:= \sum_{i=1}^{n} \log \frac{z^*_{i+1}-z^*_{i}}{z^*_{i}-z^*_{i-1}}
\end{align}
where the imaginary part of the logarithmic for each term take values in $(-\pi,\pi)$,  which is positive if the curve is turning right while negative otherwise. Notice that $|\Imaginary \log \frac{z^*_{i+1}-z^*_{i}}{z^*_{i}-z^*_{i-1}}|$ is the corner angle of the triangle opposite to $\angle z^*_{i-1}z^*_{i}z^*_{i+1}$. It can verified that the complex parameter $\tau$ and the affine parameter $c$ can be determined by 
\[
c= h_1 \quad \text{and} \quad c\, \tau= h_2.
\]

\begin{lemma}\label{lem:bounded}
	For any Delaunay cross ratio system, both $\Imaginary c$ and $\Imaginary c \tau$ are bounded by a constant depending only on the triangulation.
\end{lemma}
\begin{proof}
	 Equation \eqref{eq:hhol} implies that
	 \[
	 |\Imaginary c| = |\sum_{i=1}^{n} \Imaginary \log \frac{z^*_{i+1}-z^*_{i}}{z^*_{i}-z^*_{i-1}}| \leq n \pi
	 \]
	 where $n$ depends on the triangulation and is independent of the cross ratio systems. The same argument applies to $\Imaginary c \tau$.
\end{proof}

\subsection{Properness} \label{sec:properness}

\begin{proposition}\label{prop:proper}
	Suppose $\Theta$ is a Delaunay angle structure on the torus. Then the projection from the space of Delaunay cross ratios $P(\Theta)$ to the Teichm\"{u}ller space of the torus is a proper map.
\end{proposition}
\begin{proof}
	Suppose a sequence of Delaunay cross ratio systems $\{\cratio^{(k)}\}_{k\in \mathbb{N}}$ in $P(\Theta)$ is given such that their underlying marked conformal structures $\{\tau^{(k)}\}_{k\in \mathbb{N}}$ converge to some conformal structure $\tau^{(\infty)}$ in the upper half plane. Proposition \ref{prop:rivin} implies that the sequence of Delaunay cross ratio systems can be represented, up to a sign, by the sequence $\{(A_1^{(k)},A_2^{(k)})\}_{k\in \mathbb{N}}$ in $\mathbb{R}^2$ such that under an affine developing map, these numbers represent the scaling part of the holonomy. Namely, for $r=1,2$ and $k \in \mathbb{N}$
	\[
	A_r^{(k)} = \Re h^{(k)}_r
	\] 
	where $h$ is defined in \eqref{eq:hhol}.
	
	We claim that $\{(A_1^{(k)},A_2^{(k)})\}_{k\in \mathbb{N}}$ has a convergent subsequence. Notice that
	\begin{align}\label{eq:sequh}
		h^{(k)}_2= c^{(k)} \tau^{(k)} = \tau^{(k)} h^{(k)}_1
	\end{align}
	Taking the imaginary of both sides yields for $k\in \mathbb{N}$
	\[
	\Im h^{(k)}_2 =  \Imaginary \tau^{(k)}  \Real h^{(k)}_1 +\Real \tau^{(k)}  \Imaginary h^{(k)}_1 =A^{(k)}_1   \Imaginary \tau^{(k)}  +\Real \tau^{(k)}  \Imaginary h^{(k)}_1
	\]
	Recall from Lemma \ref{lem:bounded} that both $\{\Imaginary h^{(k)}_1\}_{k\in \mathbb{N}}$ and $\{\Imaginary h^{(k)}_2\}_{k\in \mathbb{N}}$ are bounded. Since $\{\tau^{(k)}\}_{k\in \mathbb{N}}$ converges to $\tau^{(\infty)}$ in the upper half plane, we can assume $\{\tau^{(k)}\}_{k\in \mathbb{N}}$ is bounded by considering a subsequence. Thus $\{A^{(k)}_1\}_{k\in \mathbb{N}}$ is bounded and hence has a converging subsequence with the limit denoted as $A^{(\infty)}_1$. Similarly, taking the real part of \eqref{eq:sequh} yields that there is a subsequence of $\{A^{(k)}_2\}_{k\in \mathbb{N}}$ converging to some number $A^{(\infty)}_2$ in $\mathbb{R}$. By Rivin's variational argument (Proposition \ref{prop:rivin}), there is a Delaunay cross ratio system $\cratio^{(\infty)}$ corresponding to $(A^{(\infty)}_1,A^{(\infty)}_2)$.
	
	Furthermore, one can show that both $\cratio$ and $(h_1,h_2) \in \mathbb{C}^2$ depends continuously on the parameters $(A_1,A_2)$ for the affine structures. Hence we deduce that $\{\cratio^{(k)}\}_{k\in \mathbb{N}}$ has a subsequence converging to $\cratio^{(\infty)}$ and its underlying conformal structure is $\tau^{(\infty)}$ by taking the limit of Equation \eqref{eq:sequh} as $k\to \infty$.
\end{proof}

\subsection{Immersion around non-Euclidean tori} \label{sec:immersion}

 Under an infinitesimal change of the complex projective structure with the conformal structure fixed, the change in the affine holonomy is of the form \begin{align*}
(\dot{\rho}_1(z),\dot{\rho}_2(z)) = (\dot{c} e^{c}z, \dot{c}\tau e^{c\tau}z).
\end{align*}

\begin{proposition}\label{prop:injective}
	Suppose a Delaunay cross ratio system induces a non-Euclidean affine torus ($c\neq0$). Then a discrete holomorphic quadratic differential $q$ yields an infinitesimal change in the complex projective structure while preserving the conformal structure if and only if $q \equiv 0$.
\end{proposition}
\begin{proof}

We prove by contradiction: Suppose we have a Delaunay cross ratio and we have a developing map whose holonomy has fixed points at the origin and infinity. We write its holonomy as $\rho_1(z) = e^{c}z$ and $\rho_2(z) = e^{c \tau}z$. Let $q$ be a discrete holomorphic quadratic differential as in the assumption and $\dot{z}$ be an infinitesimal deformation of the developing map given by Proposition \ref{prop:hqdharmonic}. Since $z$ is equivariant with respect to the fundamental group, we have
\begin{align*}
(\dot{z}\circ \gamma_1)_i &= d\rho_1 (\dot{z}_i) + \dot{\rho}_1 (z_i)= e^{c}\dot{z}_i+ \dot{c}e^{c} z_i \\
(\dot{z}\circ \gamma_2)_i &= d\rho_2 (\dot{z}_i) + \dot{\rho}_2 (z_i)= e^{c\tau}\dot{z}_i+ \dot{c}\tau e^{c\tau} z_i
\end{align*}

By Proposition \ref{prop:hqdharmonic}, it determines a discrete harmonic function on the universal cover of the torus with constant periods: 
 for every triangle $\{ijk\}$ we have
\begin{align*}
(u\circ \gamma_1)_i &= \Imaginary\left( \frac{(\dot{z}\circ \gamma_1)_j - (\dot{z}\circ \gamma_1)_i}{(z\circ \gamma_1)_j - (z\circ \gamma_1)_i}  + \frac{(\dot{z}\circ \gamma_1)_k - (\dot{z}\circ \gamma_1)_i}{(z\circ \gamma_1)_k - (z\circ \gamma_1)_i}  - \frac{(\dot{z}\circ \gamma_1)_j - (\dot{z}\circ \gamma_1)_k}{(z\circ \gamma_1)_j - (z\circ \gamma_1)_k} \right) \\
&= u_i + \Imaginary( \dot{c})
\end{align*}
and similarly
\begin{align*}
(u\circ \gamma_2)_i =u_i + \Imaginary( \dot{c} \tau).
\end{align*}
It yields that its conjugate harmonic function $u^*:\hat{F} \to \mathbb{R}$ must have constant periods as well. We write for $r=1,2$
\begin{align*}
(u^* \circ \gamma_r)_{ijk} = u^*_{ijk} + \delta_r
\end{align*}
where $\delta_r \in \mathbb{R}$. Using \eqref{eq:conjharmonic}, we deduce that
\[
\delta_1 = -\Real( \dot{c}), \quad \delta_2 = -\Real( \dot{c} \tau)
\]
Though $u$ is defined on the universal cover, its difference $u_j- u_i$ is well defined on the torus. We can compute its Dirichlet energy over the torus:
\[
\mathcal{E}_{T}(u):= \frac{1}{2} \sum_{ij \in E} (\cot \angle jkl + \cot \angle jil) (u_j - u_i)^2 = \frac{1}{2} \sum_{ij \in E} (u^*_{ijk} - u^*_{jil}) (u_j - u_i)
\]
where the sum is over the edges of the torus. 

Analogous to Riemann's bilinear identity in the classical theory, Bobenko and Skopenkov \cite{Bobenko2016} showed that the Dirichlet energy of such a discrete harmonic function is determined by its periods: if $u:\hat{V} \to \mathbb{R}$ is a discrete harmonic function on the universal cover and $u^*:\hat{F} \to \mathbb{R}$ is its conjugate harmonic function with constant periods $A_1,A_2 \in \mathbb{C}$ where $\Re A_r = u \circ \gamma_r - u $ and $\Im A_r = u^* \circ \gamma_r - u^* $ for $r=1,2$. Then the Dirichlet energy is given by the periods
	\[
	\mathcal{E}_T(u) = -\Imaginary (A_1 \bar{A}_2).
	\]

Substituting $A_1= -\mathbf{i}\dot{c}$ and  $A_2= -\mathbf{i}\dot{c}\tau$ where $\mathbf{i}=\sqrt{-1}$ for our case, our harmonic function $u$ has energy
\[
\mathcal{E}_T(u) =  -\Imaginary (|\dot{c}|^2 \bar{\tau}) = |\dot{c}|^2 \Imaginary ( \tau) \geq 0.
\]

We claim that this energy is not achievable by a discrete harmonic function unless $\dot{c}=0$. Notice that $u$ is defined on the vertices of a triangulation. We can extend it piecewisely over faces to obtain a piecewise linear function $\tilde{u}: \tilde{M} \to \mathbb{R}$, which has the same periods as $u$.
Its Dirichlet energy in the classical theory is defined by:
\[
\mathcal{E}(\tilde{u}) := \iint_M |\grad \tilde{u}|^2 dA
\]
Using the property of the cotangent Laplacian \cite{Pinkall1993}, we have \[\mathcal{E}(\tilde{u}) = \mathcal{E}_T(u) =|\dot{c}|^2 \Imaginary ( \tau) \]
We compare this energy with a smooth harmonic function: Consider the function $u^{\dagger}:=\Re (-\mathbf{i} \frac{\dot{c}}{c} \log z)$ defined on the universal cover of $\mathbb{C}-{0}$. Pulled back by the developing map, it defines a harmonic function on the universal cover of the torus. Indeed it is a smooth a harmonic function with constant periods for the torus :
\begin{align*}
u^{\dagger} \circ \gamma_1(z) - u^{\dagger}(z) &= \Re (-\mathbf{i}\frac{\dot{c}}{c} \log (e^c z)) - \Re (-\mathbf{i}\frac{\dot{c}}{c} \log z) = \Re (-\mathbf{i}\dot{c}) = \Imaginary (\dot{c}) \\
u^{\dagger} \circ \gamma_2(z) - u^{\dagger}(z) &= \Re (-\mathbf{i}\frac{\dot{c}}{c} \log (e^{c\tau} z)) - \Re (-\mathbf{i}\frac{\dot{c}}{c} \log z) = \Re (-\mathbf{i}\dot{c}\tau)=\Imaginary (\dot{c} \tau) 
\end{align*}  
whose conjugate harmonic function $\Imaginary (-\mathbf{i}\frac{\dot{c}}{c} \log z)$ has periods $-\Real (\dot{c})$ and $-\Real (\dot{c}\tau) $ similarly. Using the Riemann bilinear identity from the smooth theory, the Dirichlet energy of $u^{\dagger}$ on the torus equals to $\mathcal{E}(u^{\dagger}) =  |\dot{c}|^2 \Imaginary ( \tau)$. However, notice that $u^{\dagger}$ is the unique minimizer (up to a constant) of the classical Dirichlet energy among piecewise smooth functions with periods $\Imaginary(\dot{c})$ and $\Imaginary (\dot{c}\tau)$, we have if $\dot{c} \neq 0$
\[
  |\dot{c}|^2 \Imaginary ( \tau) = \mathcal{E}(u^{\dagger}) < \mathcal{E}(\tilde{u}) =   |\dot{c}|^2 \Imaginary ( \tau)
\]
which is a contradiction. It implies $\dot{c}=0$ and the discrete harmonic function $u$ is constant by the maximum principle. Using Proposition \ref{prop:hqdharmonic}, we have $q\equiv0$.
\end{proof}

\begin{proof}[Proof of Theorem \ref{thm:covering}]
	For a Delaunay angle structure $\Theta$ on the torus, Proposition \ref{prop:injective} shows that $d(\pi\circ f)$ is injective over $P(\Theta)  \backslash \{\cratio_0\} $ and hence $ \pi\circ f:P(\Theta) \backslash \{\cratio_0\} \to  \mathcal{T}(M)  \backslash \{\tau_0\}$ is a local homeomorphism. On the other hand, Proposition \ref{prop:proper} implies the map is proper. Thus $\pi\circ f:P(\Theta) \backslash \{\cratio_0\} \to  \mathcal{T}(M) \backslash \{\tau_0\}$ is a finite-sheet covering.
\end{proof}

\begin{remark}
	For the Euclidean torus, the corresponding periods of discrete harmonic functions are not constant but linear functions. Hence their Dirichlet energy is not well defined on the torus and the argument in Proposition \ref{prop:injective} could not apply. 
\end{remark}

\section{Examples} \label{sec:examples}

\subsection{Non-Delaunay cross ratio system: Jessen's orthogonal icosahedron}  \label{sec:jessen}

	 Jessen's icosahedron \cite{Jessen1967} is combinatorially a regular icosahedron with some edges flipped (see Fig. \ref{fig:Jessen} left). Its vertices lie on a sphere and all dihedral angles are either $\pi/2$ or $3\pi/2$. It is a non-convex triangulated sphere that is infinitesimally flexible, i.e. it admits an infinitesimal deformation of vertices such that edge lengths are preserved.
	
	 It is known in \cite{Lam2015} that the non-trivial infinitesimal isometric deformation induces a non-vanishing holomorphic quadratic differential on its stereographic image (see Fig. \ref{fig:Jessen} right). Notice that the orthogonal icosahedron consists two kind of edges: Every vertex is connected to four short edges and one long edge. We define $q_{ij}=1$ on the short edges and $q_{ij}=-4$ on the long edges. One immediately have $\sum_j q_{ij}=0$ for every vertex $i$. Denoting $z$ the stereographic image of the vertices, one can check $\sum_j q/(z_j-z_i) =0$ around every vertex and $q$ is a holomorphic quadratic differential. We thus obtain a non-Delaunay triangulation of a sphere that carries a non-trivial holomorphic quadratic differential.
	
	\begin{figure}[!h]
		\begin{minipage}{0.45\textwidth}
			\centering
			\tdplotsetmaincoords{73.5}{22}
			\begin{tikzpicture}[scale=1,tdplot_main_coords]
			\coordinate (A) at (1,-2,0); 1
			\coordinate (E) at (-1,-2,0);2
			\coordinate (D) at (2,0,1);3
			\coordinate (H) at (2,0,-1);4
			\coordinate (B) at (0,-1,2);5
			\coordinate (I) at (0,1,2);6
			\coordinate (F) at (-2,0,1);7
			\coordinate (G) at (-2,0,-1);8
			\coordinate (bA) at (1,2,0);9
			\coordinate (bE) at (-1,2,0);10
			\coordinate (C) at (0,-1,-2);11
			\coordinate (bI) at (0,1,-2);12
			
			\draw (A)--(B)--(C)--(A);
			\draw (B)--(E)--(C);
			\draw (A)--(D)--(B);
			\draw (B)--(F)--(E);
			\draw (E)--(G)--(C);
			\draw (A)--(H)--(C);
			\draw (D)--(I)--(F);
			\draw[dashed] (F)--(D);
			\draw (A)--(bA)--(D);
			\draw (H)--(bA);
			\draw[dashed] (E)--(bE);
			\draw[dashed] (I)--(bA);
			\draw[dashed] (I)--(bE);
			\draw[dashed] (I)--(bI);
			\draw[dashed] (bI)--(H);
			\draw[dashed] (bI)--(G);
			\draw[dashed] (bI)--(bA);
			\draw[dashed] (bI)--(bE);
			\draw[dashed] (H)--(G);
			\draw[dashed] (G)--(bE);
			\draw[dashed] (F)--(bE);
			\end{tikzpicture}
		\end{minipage}
		\begin{minipage}{0.45\textwidth}
			\centering
			\includegraphics[width=0.7\textwidth]{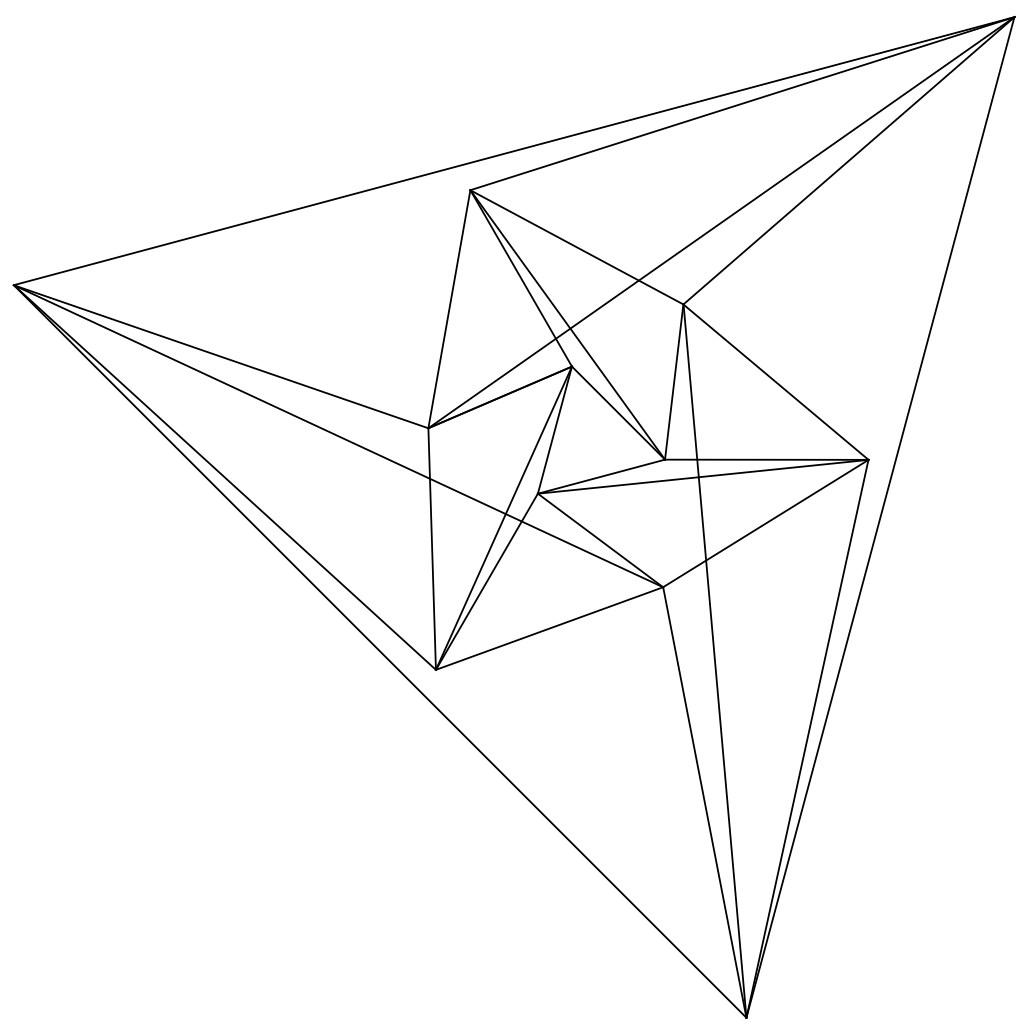}
		\end{minipage}
		\caption{Jessen's orthogonal icosahedron (left) and its stereographic projection (right).}
		\label{fig:Jessen}
	\end{figure}
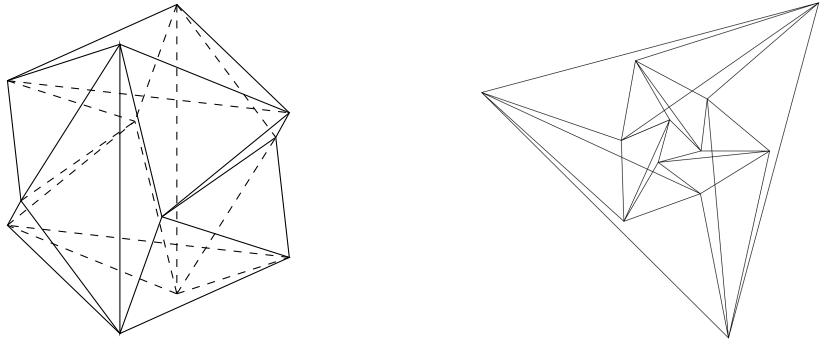

	Jessen's orthogonal icosahedron provides a construction to obtain a non-Delaunay cross ratio system on a given surface with the dimension of discrete holomorphic quadratic differentials as large as one desires. To see this, suppose a cross ration system $\cratio$ is given on a triangulation $T$ of a surface and consider its developing map. Pick a face $\phi$ of $T$. Then we attach a M\"{o}bius image of Jessen's icosahedron to all the lifts of $\phi$ on the universal cover under the developing map. In this way, we obtain a new triangulation $\tilde{T}$ of the given surface together with a new cross ratio system $\tilde{cr}$. It can be verified that  $\dim Q^{\mathbb{R}}(\tilde{\cratio}) = \dim Q^{\mathbb{R}}(\cratio) +1$.

\subsection{One-vertex triangulated torus} \label{sec:one-vertex}
We consider the one-vertex triangulation of the torus, which has three edges and two faces. On the universal cover, each vertex is connected to six edges and the equations in Definition \ref{def:crsys} becomes:
\begin{align*}
1&= (\cratio_1 \cratio_2 \cratio_3)^2 \\
0&= (\cratio_1 + \cratio_1 \cratio_2+ \cratio_1 \cratio_2 \cratio_3) ( 1 +  \cratio_1 \cratio_2 \cratio_3)
\end{align*}
There are two cases for the solutions.

\textbf{Case (a): $\cratio_1 \cratio_2 \cratio_3=-1$}. One can verify that the holonomy $\rho_1$, $\rho_2$ share the same fixed points and can be normalized as complex affine transformations. 

On the other hand, $q_1,q_2,q_3$ on the edges yields a holomorphic quadratic differential if and only if $q_1+q_2+q_3=0$. It implies $	\dim_{\mathbb{C}} Q^{\mathbb{C}}=2$ and  $	\dim_{\mathbb{R}} Q^{\mathbb{R}}=2$. In particular, the space of the solution with prescribed $| \cratio|$ or $\Arg \cratio$ is a manifold of real dimension $2$.

\textbf{Case (b): $\cratio_1 \cratio_2 \cratio_3=1$ and $\cratio_1 + \cratio_1 \cratio_2+ \cratio_1 \cratio_2 \cratio_3=0$ }. One can show that $\cratio_1, \cratio_2, \cratio_3$ are of the form $b, -(b+1)/b, -1/(1+b)$ for some complex number $b$ and it is non-Delaunay for any choice of $b$. To see this, pick a vertex $z_0$ and denote its neighboring vertices in $\hat{M}$ as $z_1,z_2,z_3,z_{\tilde{1}},z_{\tilde{2}},z_{\tilde{3}}$ in counterclockwise orientation. One can show that $z_i=z_{\tilde{i}}$ for $i=1,2,3$ and furthermore the holonomy exchanges the fixed points of each other and is of type (III) in Section \ref{subsec:holonomy}, which does not appear for the torus in the smooth theory. In fact, the gluing construction in Proposition \ref{prop:cp1} yields a surface with boundary in this case.

On the other hand, $\{q_1,q_2,q_3\}$ is in the kernel of $D(\Phi\circ \exp)$ (see Section \ref{sec:hqd}) if it is in the form of
\begin{align*}
q_1 &= (b-1) q_2 \\
q_3 &= -b \, q_2
\end{align*}
up to scaling. It implies 	$\dim_{\mathbb{C}} Q^{\mathbb{C}} =1$. For quadratic differentials, since $q$ is required be to real-valued, we have
\begin{align*}
\dim_{\mathbb{R}} Q^{\mathbb{R}} &= 1 \quad \text{if } b\in \mathbb{R} \\
\dim_{\mathbb{R}} Q^{\mathbb{R}} &= 0 \quad \text{otherwise}
\end{align*}

\section{Discussion and open questions} \label{sec:discussion}

In this section, we discuss some open questions that extend the conjecture of Kojima, Mizushima and Tan \cite{Kojima2003}. 

\subsection{Surfaces with genus $g$}

Theorem \ref{thm:holo} and \ref{thm:covering} should have counterparts for surfaces other than the torus. The first question is whether $P(\Theta)$ is a manifold for a given Delaunay angle structures. It is equivalently to asking if $Q^{\mathbb{R}}(\cratio)$ has the right dimension as the Teichm\"{u}ller space. For the torus $(g=1)$, it is answered in Theorem \ref{prop:kernel}. For the sphere $(g=0)$, it is trivial.

\begin{proposition}
	Suppose $\Theta$ is a Delaunay angle structure on a triangulated sphere ($g=0$). Then along $P(\Theta)$, we have 
	\[
		\dim_{\mathbb{R}} Q^{\mathbb{R}}(\cratio) =0.
	\]
	and $P(\Theta)=\{\cratio_0\}$ consists of only one element.
\end{proposition}
\begin{proof}
	The fact that $P(\Theta)=\{\cratio_0\}$ consists of only one element is asserted by Rivin's variational method \cite{Rivin}.
	
	To see $\dim_{\mathbb{R}} Q^{\mathbb{R}}(\cratio) =0$, we can use discrete harmonic functions again. Consider a developing map of the underlying complex projective structure, it covers the sphere exactly once since there is only one complex projective structure on the sphere. We normalize the developing map such that only one circumdisk contains infinity and remove the corresponding face. What remains is an embedded triangulated disk in the plane with three boundary vertices. Notice that by Dirichlet's principle, any discrete harmonic function on a domain with three boundary vertices must be a linear function.
	
	Thus, for any $q \in Q^{\mathbb{R}}(\cratio)$, the corresponding discrete harmonic function from Proposition \ref{prop:hqdharmonic} is a linear function on the triangulated disk. Hence $q\equiv 0$.
\end{proof}

\begin{proposition}
	The space of cross ratio systems (not necessary to be Delaunay) on a triangulated sphere $M=(V,E,F)$ is a manifold of complex dimension $|V|-3$.
\end{proposition}
\begin{proof}
	We can put the vertices on the sphere arbitrarily as long as the endpoints of every edge are distinct. Thus, the configuration space modulo M\"{o}bius transformations has complex dimension $|V|-3$. 
	
	Similarly, one can assign a vector field to each vertex arbitrarily for a given developing map. Modulo infinitesimal M\"{o}bius transformations, it implies $\dim_{\mathbb{C}} Q^{\mathbb{C}}(\cratio)=|V|-3$.
\end{proof}

\begin{conjecture}\label{thm:dim}
	
	For a Delaunay cross ratio system $\cratio$ on  a closed triangulated surface with genus $g$, we have
	\begin{align*}
	\dim_{\mathbb{C}} Q^{\mathbb{C}}(\cratio) = \begin{cases}
	|V| + 1  \quad &\text{if } g =1 \\
	|V| + 6g- 6 \quad &\text{if } g \geq 2
	\end{cases}
	\end{align*}
	and
	\begin{align*}
	\dim_{\mathbb{R}} Q^{\mathbb{R}}(\cratio) = 
	6g - 6 \quad \text{if } g \geq 2
	\end{align*}
\end{conjecture}

\begin{conjecture} \label{conj:full}
	For any Delaunay angle structure $\Theta$ on  a closed triangulated surface with genus $g>1$, the holonomy map
       \[ \hol: P(\Theta) \to \mathcal{X}(M) \]
       is an embedding of a manifold homeomorphic to $\mathbb{R}^{6g-6}$. Furthermore its projection to the Teichm\"{u}ller space is diffeomorphic.
\end{conjecture}

\subsection{Prescribed hyperbolic metrics}

Indeed, every Delaunay cross ratio system is associated with a locally convex pleated surface in hyperbolic 3-space, where the shearing coordinate between neighboring facets is $\Re \log cr$ and the dihedral angle is captured by $\Arg \cratio = \Im \log \cratio$ (See \cite{Bobenko2010} for the construction). The Delaunay condition is equivalent to the surface being locally convex. The space $P(\Theta)$ describes all these pleated surfaces with prescribed dihedral angles. In contrast, one can consider those locally convex pleated surfaces with a prescribed hyperbolic metric $d$. Whenever a triangulation $\mathcal{T}$ of the surface is fixed, the hyperbolic metric $d$ is described by the shear coordinates $X_{d,\mathcal{T}}:E \to \mathbb{R}$. Thus the space of Delaunay cross ratio systems that induce pleated surfaces with hyperbolic metric $d$ can be written as
\[
P(d) = \cup_{\mathcal{T}} P(X_{d,\mathcal{T}})
\]  
where $P(X_{d,\mathcal{T}})$ consists of all the Delaunay cross ratio systems $\cratio$ on $\mathcal{T}$ with $\log |\cratio|\equiv X_{d,\mathcal{T}}$ and the union is taken over all triangulations of the surface with the same vertex set \cite{Bobenko2010,Gu20182}. This space is non-empty by the discrete uniformization theorem. If the underlying surface is a torus, then the argument in Section \ref{sec:cp1tori} and \ref{sec:immersion} yields that $P(d)$ is a surface and its projection to the Teichm\"{u}ller space is an immersion. It is interesting to consider Conjecture \ref{conj:full} with prescribed hyperbolic metrics in place of prescribed dihedral angles.

\section*{Acknowledgment}

The author would like to thank Feng Luo and Richard Schwartz for fruitful discussions and Masashi Yasumoto for comments on the
draft.

\appendix

\section{Rigidity on complex projective structures}

Here we provide an alternative proof that two Delaunay cross ratio systems inducing the same complex projective structure on the torus must be identical. Though this result can be deduced from Rivin's variational method (Section \ref{sec:Rivin}), the following proof involves the maximum principle which might be inspiring for surfaces with genus $g\geq2$. 
\begin{theorem}
	Let $cr, \tilde{cr}:E \to \mathbb{C}$ be two Delaunay cross ratios system on a triangulated torus with $\Arg(cr) = \Arg(\tilde{cr})$. Suppose they admit developing map $z, \tilde{z}$ with affine holonomy $\rho_r(z)=\alpha_r z + \beta_r$ and $\tilde{\rho}_r(z)=\tilde{\alpha}_r z + \tilde{\beta}_r$ such that $|\alpha_r|=|\tilde{\alpha}_r|$ for $r=1,2$. Then the developing maps $z,\tilde{z}$ differ by a similarity and we have $\cratio= \tilde{\cratio}$.
\end{theorem}
\begin{proof}
	We claim that the corresponding circumscribed circles differ by a global scaling. It follows from the maximum principle as follows:
	
	Whenever there is an edge with $\Arg \cratio =0$, we remove the edge and merge the neighboring faces since the corresponding circumcircles coincide. By the definition of Delaunay cross ratio systems, we obtain a cell decomposition $(V,E,F)$ of $M$. The Delaunay condition further implies that under the developing map each face is a convex polygon. Indeed, each face is obtained by merging triangles and so we can argue by induction: If the face consists of one triangle, then it is convex obviously. If there are two triangles $\{ijk\},\{ilj\}$, then $\Arg \cratio_{ij} = 0$ implies they form a convex quadrilateral. Suppose we have a convex cyclic k-gon and attach a new triangle $\{jil\}$ to an edge $\{ij\}$ of the k-gon, then the new vertex $z_l$ must lie on the arc between $z_i,z_j$ not containing the other vertices since $\Arg \cratio_{ij} = 0$. Hence the resulting $(k+1)$-gon is convex cyclic.
	
	We denote the $R:\hat{F}\to \mathbb{R}_{>0}$ and $\tilde{R}:\hat{F}\to \mathbb{R}_{>0}$ the radii of the circumcircles under the developing map $z$ and $\tilde{z}$ of the universal cover $\tilde{M}$. We consider the ratio of the radii $\sigma:= \tilde{R}/R$. Since the holonomy is affine, the ratio $\sigma$ is periodic. Hence, $\sigma$ has a local maximum on a face $f_0$. We denote the neighboring faces as $f_1,f_2,\dots, f_k$ and the intersection angles of the circumcircles as $\phi_1,\phi_2, \dots, \phi_k$. Note $\phi_i = \Arg \cratio_i \in (0,\pi)$.
	
	We focus on the developing map $z$. Suppose $z_i z_{i+1}$ is the common chord shared by the circles at $f_0$ and $f_i$. We write the circumcenter of $f_0$ as $O$ and denote
	$2 \alpha_i$ the angle at the center from $Oz_i$ to $Oz_{i+1}$ in counterclockwise orientation (See Figure \ref{fig:centralangles}). Then
	\[
	R_0 \sin \alpha_i = R_i \sin (\phi_i - \alpha_i) = R_i( \sin \phi_i \cos \alpha_i - \cos \phi_i \sin \alpha_i)
	\]
	Hence
	\[
	\cot \alpha_i = ( \frac{1}{\sin \phi_i}(\frac{R_0}{R_i} + \cos \phi_i))
	\]
	Since $\sigma_0 \geq \sigma_i $, we have
	\[
	\cot \alpha_i =  \frac{1}{\sin \phi_i}( \frac{\sigma_i \tilde{R}_0}{\sigma_0 \tilde{R}_i} + \cos \phi_i) \leq  \frac{1}{\sin \phi_i}( \frac{\tilde{R}_0}{\tilde{R}_i} + \cos \phi_i) = \cot \tilde{\alpha}
	\]
	Note $\cot(\cdot)$ is monotone decreasing on $(0, \pi)$. Hence
	$\alpha_i \geq \tilde{\alpha}_i$.
	\begin{figure} \centering
		\begin{minipage}{0.49\textwidth}
			\includegraphics[width=1\textwidth]{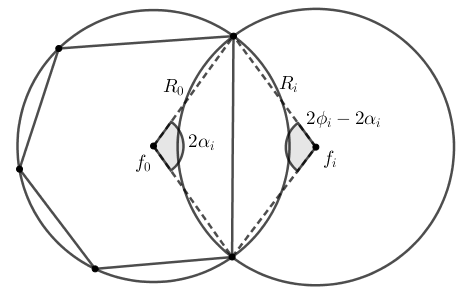}
		\end{minipage}
		\begin{minipage}{0.49\textwidth}
			\includegraphics[width=0.9\textwidth]{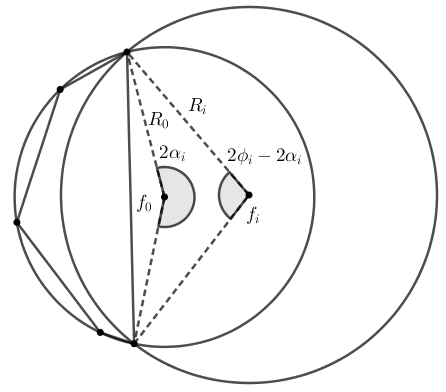}
		\end{minipage}
		\caption{Central angles. Left: Center within the polygon. Right: Center outside the polygon.}
		\label{fig:centralangles}
	\end{figure}
	
	On the other hand, because the vertices of the face $f_0$ are in cyclic order on the circle, the sum of the central angles are equal to $2\pi$, we have
	\[
	2\pi = \sum 2 \alpha_i \geq \sum 2 \tilde{\alpha}_i = 2 \pi.
	\]
	Thus all the equalities should hold and $\sigma_i = \sigma_0$ for $i=1,2,\dots,k$. Hence the maximum principle holds and $\sigma$ must be constant.
\end{proof}

\bibliographystyle{hplain}
\bibliography{holomorphicquad}

\end{document}